\documentclass[12pt,draft,leqno]{article}
\usepackage{amssymb,eucal,latexsym}
\usepackage{amsmath,amsthm,latexsym,amscd,amsbsy,fleqn,graphicx}
\usepackage{xy}
\xyoption{all}

\newtheorem{theorem}{Theorem}[section]
\newtheorem{lm}[theorem]{Lemma}
\newtheorem{exa}[theorem]{Example}

\newtheorem{cor}[theorem]{Corollary}
\newtheorem{pro}[theorem]{Proposition}
\newtheorem{defi}[theorem]{Definition}

\newtheorem{nota}[theorem]{Notation}

\newtheorem{rem}[theorem]{Remark}

\newtheorem{fact}[theorem]{Fact}

\newtheorem{nist}[theorem]{}

\newcommand{\timplies}{\text{ implies }}

\newcommand{\df}{\ensuremath{\overset{\mathrm{df}}{=}}}
\newcommand{\tiff}{if and only if \ }

\newcommand{\ult}{\operatorname{Ult}}

\def\p{\varphi}
\def\a{\alpha}

\def\d{\delta}
\def\ep{\varepsilon}
\def\g{\gamma}
\def\GA{\Gamma}

\def\s{\sigma}

\def\ups{\upsilon}
\def\up{\upsilon}

\def\fs{\hat{f}}

\def\lra{\longrightarrow}
\def\lla{\longleftarrow}

\def\sbe{\subseteq}

\def\stm{\setminus}
\def\ems{\emptyset}
\def\nes{\neq\emptyset}

\def\cuk{\,\check{}\,}

\def\ti{\tilde}

\def\ex{\exists}

\def\we{\wedge}

\def\ap{^{\prime}}
\def\inv{^{-1}}
\def\st{\ |\ }

\def\nin{\not\in}

\def\card #1{\vert #1 \vert}

\def\1{{\bf 1}}
\def\2{\mbox{{\bf 2}}}
\def\3{\mbox{{\bf 3}}}

\def\AA{{\cal A}}
\def\BB{{\cal B}}
\def\CC{{\cal C}}

\def\DD{{\cal D}}

\def\PP{{\cal P}}

\def\TT{{\cal T}}

\newcommand{\RC}{\operatorname{RC}}
\newcommand{\RO}{\operatorname{RO}}

\newcommand{\CO}{\operatorname{CO}}
\newcommand{\Mor}{\operatorname{Mor}}

\def\DHC{{\bf DeV}}
\def\Fed{{\bf Fed}}
\def\HC{{\bf CHaus}}

\def\Bool{{\bf Bool}}
\def\Stone{{\bf Stone}}

\def\CBool{{\bf CBool}}

\def\cod{{\rm cod}}

\def\int{\mbox{{\rm int}}}

\def\cl{\mbox{{\rm cl}}}
\def\CL{\mbox{{\rm Clust}}}

\def\Ult{\mbox{{\rm Ult}}}

\def\doc{\hspace{-1cm}{\em Proof.}~~}
\def\sq{\hspace*{\fill} \hbox{\vrule\vbox{\hrule\phantom{o}\hrule}\vrule}}
\def\sqs{\sq \vspace{2mm}}

\def\Top{{\bf Top}}

\def\baf{\hat{f}}

\def\ECH{{\bf ECH}}

\def\tcx{t_X^C}
\def\tcy{t_Y^C}
\def\tcx0{t_{(X,X_0)}}
\def\tcy0{t_{(Y,Y_0)}}

\def\di{\diamond}

\def\bU0{\bar{U}=(U^0,(U^i,U^{ci})_{i\in\omega})}

\def\bV0{\bar{V}=(V^0,(V^i,V^{ci})_{i\in\omega})}

\title{{\LARGE\bf Extensions of dualities and a new}\\
\vspace{0.2cm}
{\LARGE\bf approach to the Fedorchuk duality}\\
\vspace{0.5cm}
{\large\bf G. Dimov, E. Ivanova-Dimova
and W. Tholen}\thanks{The first two authors were supported under
contract no. 80-10-107/19.04.2018 "Dual categories"  of Sofia University $``$St. Kl. Ohridski''. The third author acknowledges the support under Discovery Grant no. 501260 of the Natural Sciences and Engineering Council of Canada.}
\\
\vspace{0.2cm}
{\footnotesize\rm Faculty of Mathematics and Inf.,  Sofia University,}
{\footnotesize\rm 5 J. Bourchier Blvd., 1164 Sofia, Bulgaria}\\
{\footnotesize\rm Department of Mathematics and Stats., York University,}
 {\footnotesize\rm Toronto, Ontario, M3J 1P3, Canada}
\\
\vspace{0.5cm}
{\large\em Dedicated to Professor Alexander Vladimirovich Arhangel'ski\u{\i} on the occasion of his
80th birthday}}

\author{}

\date{}

\begin{document}

\maketitle

\begin{abstract}
Applying a general categorical construction for the extension of dualities,
we present a new proof of the Fedorchuk duality between the category of compact Hausdorff spaces with their quasi-open mappings and the category of
complete normal contact algebras with suprema-preserving Boolean homomorphisms which reflect the contact relation.


\end{abstract}



\footnotetext[3]{{\footnotesize {\em E-mail addresses:}
gdimov@fmi.uni-sofia.bg, elza@fmi.uni-sofia.bg, tholen@mathstat.yorku.ca}.}

\section{Introduction}

The celebrated Stone Duality Theorem \cite{ST} shows that the entire information about a zero-dimensional compact Hausdorff space (= {\em Stone space}) $X$  is, up to homeomorphism, contained in its Boolean algebra $\CO(X)$ of all clopen (= closed and open) subsets of $X$. Likewise, all information about the continuous maps between two such spaces $X$ and $Y$ is encoded by the Boolean homomorphisms between the Boolean algebras $\CO(Y)$ and $\CO(X)$. It is natural to ask whether a similar result holds for all compact Hausdorff spaces  and continuous maps between them. The first candidate for the role of the Boolean algebra $\CO(X)$
under such an extension seems to be the Boolean algebra $\RC(X)$ of all regular closed subsets of a compact Hausdorff space $X$ (or,  its isomorphic copy $\RO(X)$, which collects all regular open subsets of $X$), but it fails immediately: indeed, as it is well-known,  $\RC(X)$ is isomorphic to $\RC(EX)$, where $EX$ is the absolute of $X$. However, in 1962, de Vries \cite{deV} showed that, if we regard the Boolean algebra $\RC(X)$ together with the relation $\rho_X$ on $\RC(X)$, defined by
$$F\rho_X G\Leftrightarrow F\cap G\nes,$$
then the pair $(\RC(X),\rho_X)$ determines uniquely (up to homeomorphism) the compact Hausdorff space $X$. Moreover, with the help of some special maps between two pairs $(\RC(X),\rho_X)$ and $(\RC(Y),\rho_Y)$, where $X$ and $Y$ are compact Hausdorff spaces, one can reconstruct all continuous maps between $Y$ and $X$. He gave an algebraic description of the pairs $(\RC(X),\rho_X)$ as pairs $(A,C)$, formed by a complete Boolean algebra $A$ and a relation $C$ on $A$, satisfying some axioms, and he also described algebraically the needed special maps of such pairs. In this way he obtained the category $\DHC$ and its dual equivalence with the category $\HC$ of compact Hausdorff spaces and continuous maps. In fact, de Vries did not use the relation $\rho_X$ as mentioned above, but its ``dual", that is, the relation $F\ll_X G$ defined by
$(F\ll_X G\Leftrightarrow F(-\rho_X) G^*)$, where $G^*$ is the Boolean negation of $G$ in $\RC(X)$; equivalently,  $$F\ll_X G\Leftrightarrow F\sbe \int_X(G).$$
Now known as {\em de Vries algebras},
he originally called the abstract pairs $(A,\ll)$ {\em compingent algebras}. The axioms for the relation $C$ (respectively, $\ll_C$) on $A$ are precisely the axioms for Efremovi\v{c} proximities \cite{EF}, with only one exception: instead of Efremovi\v{c}'s separation axiom, which refers to the points of the space in question, de Vries introduced what is now called the {\em extensionality axiom} (see \cite[Lemma 2.2, p.215]{DV1} for a motivation for this terminology). Since Efremovi\v{c} proximities are relations on the Boolean algebra $({\rm P}(X),\sbe)$ of all subsets of a set $X$, de Vries algebras may be regarded as point-free generalizations of the Efremovi\v{c} proximities. Nowadays the pairs $(A,C)$, where $A$ is a Boolean algebra and $C$ is a proximity-type relation on $A$, attract the attention not only of topologists, but also of logicians and theoretical computer sciencists. Amongst the many generalizations of de Vries algebras, the most popular ones are the so-called {\em RCC systems (Region Connection Calculus)} of Randell-Cui-Cohn \cite{RCC}.
Their generalizations include the {\em contact algebras} (introduced in \cite{DV1,DV2}),  which are point-free analogues of the \v{C}ech proximity spaces,
 and {\em precontact algebras}, defined  independently and almost simultaneously, but in a completely different form, by S. Celani \cite{Celani} (for the needs of logic) and by I. D\"{u}ntsch and D. Vakarelov \cite{DUV} (for the needs of theoretical computer science). These and the RCC systems are very useful notions in the foundations of
 artificial intelligence,  geographic information systems, robot navigation, computer-aided design, and more (see \cite{ch01}, \cite{cosit2009} or \cite{ww12} for details), as well as  in logic,  namely in {\em spatial logics}
 \cite{A}  (called sometimes {\em logics of space}\/).

   A relation $C$ on a Boolean algebra $A$ which satisfies the de Vries axioms (corresponding to the relation $\rho_X$ above) is called  a {\em  normal contact relation}, and the pair $(A,C)$ then becomes a {\em  normal contact algebra} (briefly, an NCA) \cite{DV1}. In other words, the de Vries algebras ``in $\rho_X$-form" are precisely the complete NCAs.
   De Vries \cite{deV} noted that his dual equivalence $\Psi^a:\DHC\lra\HC$
  is an extension of the restriction $T:\CBool\lra\ECH$ of Stone's dual equivalence $S^a:\Bool\lra\Stone$; here $\Bool$ denotes the category of  Boolean algebras and Boolean homomorphisms  and
  $\CBool$ its full subcategory  of complete Boolean algebras; $\Stone$ is the category of  Stone spaces and continuous maps, and
  $\ECH$ denotes its full subcategory of extremally disconnected compact Hausdorff spaces. Therefore, {\em the objects of the category $\DHC$ are precisely $``$the structured $\CBool$-objects $(A,C)$".}

 Using the de Vries duality, in \cite[Theorem 8.1(1)]{Bezh} Bezhanishvili proved that, if $A$ is a complete Boolean algebra, then there exists a bijective correspondence between the set of all normal contact relations $C$ on $A$ and the set of all (up to homeomorphism) Hausdorff irreducible images of the Stone dual $S^a(A)$ of $A$. Hence, {\em the objects of de Vries' category $\DHC$ may be regarded as pairs $(A,p)$, where $A$ is a $\CBool$-object and $p:S^a(A)\lra X$ is an irreducible map onto a Hausdorff space $X$, so that $p$ is a special $\HC$-morphism;} in fact, $p$ is a {\em projective cover of $X$}.
With the structure of the objects presented in map form, we are ready to formulate the principal problem of this paper in categorical terms.

Let $T:\AA\lra\BB$
 be a dual equivalence between two categories $\AA$ and $\BB$, and $\BB$ be a full subcategory of a category $\CC$. Then it is not at all surprising that one can construct a category $\DD$ containing $\AA$ as a full subcategory, and a dual equivalence $\tilde{T}:\DD\lra\CC$ extending $T$ along the inclusion functors $I$ and $J$, as in the diagram
$$\xymatrix{\DD
\ar[r]^{\tilde{T}}\ar@{}[rd]|{\rm{}} & \CC
\\\AA\ar@{^(->}[u]^{J}\ar[r]^{T} & \BB.\ar@{^(->}[u]_{I}}$$
Inside $\CC$, one may simply replace $\BB$ by $\AA$ and adjust the composition using the dual equivalence $T$ to obtain the category $\DD$! This \emph{ad-hoc} procedure, however, does not make for a naturally described category $\DD$, since the definition of the hom-sets of $\DD$ changes with the two types of objects involved.
The principal goal of this paper is therefore to model the objects of a suitable extension category $\DD$ of $\AA$ dually equivalent to $\CC$ in a {\em natural} way, as $\AA$-objects provided with a structure that gives them a strong algebraic flavour. Our comments on the de Vries duality suggest to consider as objects of $\DD$ the pairs $(A,p)$, with $A$ an $\AA$-object and $p$ a $``$special" $\CC$-morphism with domain $T(A)$. Being $``$special" may be described as lying in a given class $\PP$ of $\CC$-morphisms  satisfying suitable axioms, which suffice to establish a category $\DD$ with a dual equivalence $\ti{T}:\DD\lra\CC$. In executing this program, our principal target is the Fedorchuk duality. The application of the general setting to other dualities, including the de Vries duality, is considered in the sequel \cite{DDT} to this paper.

The $\DHC$-morphisms are quite unusual and not very convenient to work with, since they are not Boolean homomorphisms, and their $\DHC$-composition is not the set-theoretical composition of functions. But, as Fedorchuk \cite{F} noted, the suprema-preserving Boolean homomorphisms which reflect the contact relation, are $\DHC$-morphisms and, moreover, their $\DHC$-composition coincides with the usual set-theo\-re\-tical composition of functions. He therefore considered the (non-full) subcategory $\Fed$ of $\DHC$ with the same objects, but with morphisms only these $``$good" $\DHC$-morphisms. He showed that the restriction $\Phi^a$ of the de Vries dual equivalence $\Psi^a:\DHC\lra\HC$
to his subcategory $\Fed$  sends $\Fed$ onto the subcategory $\HC_{{\rm qop}}$ of compact Hausdorff spaces and quasi-open mappings of the category $\HC$, obtaining in this way his duality theorem.

Here is a brief explanation of how the Fedorchuk duality has motivated and, in turn, be derived from, our categorical approach. Let $\CBool_{\rm{sup}}$ denote the category of  complete Boolean algebras with their suprema-preserving Boolean homomorphisms, and $\ECH_{\rm{op}}$ the category of extremally disconnected compact Hausdorff spaces and their open maps. With $\AA\df \CBool_{\rm{sup}}$, $\BB\df \ECH_{\rm{op}}$ and $\CC\df \HC_{{\rm qop}}$, $\BB$ is obviously a full subcategory of $\CC$, and it follows immediately from a result obtained in \cite[Corollary 3.2(c)]{D-PMD12} (see also \cite[Corollary 2.4(c)]{D-a0903-2593}), that the restriction $T$ of the Stone dual equivalence $S^a:\Bool\lra\Stone$ to the category $\AA$ is a dual equivalence between the categories $\AA$ and $\BB$.  Consider now the class $\PP$ of all irreducible maps between compact Hausdorff spaces with domain in $\BB$; these are well known to be morphisms in $\CC$. Hence, $\Fed$-objects (which coincide with $\DHC$-objects) may be regarded as the pairs $(A,p)$ with $p:T(A)\lra C$ in the class $\PP$. We prove here that the category $\BB$ is a coreflective subcategory of the category $\CC$. This means equivalently that, if  $p:EX\lra X$ and $p\ap:EX\ap\lra X\ap$ are projective covers ({\em i.e.}, $p,p\ap\in\PP$), then, for every quasi-open mapping $f:X\lra X\ap$ there exists a unique open mapping $\fs:EX\lra EX\ap$ such that $f\circ p=p\ap\circ\fs$. In the particular case when $f$ is a surjection, this follows from the results of Henriksen and Jerison \cite{HJ} and Bereznitskij (cited in \cite{PS}), but we were not able to find a reference for this result without the surjectivity assumption.

Our categorical extension for a dual equivalence relies on five basic conditions on the abstract class $\PP$ (see Section 3). In its basic form it is formulated in Proposition \ref{mere duality}; Theorem \ref{what about S} gives a more comprehensive and sophisticated version of the extension theorem, exhibiting all functorial relations of the extended dual equivalence with the given data, but it requires a slight strengthening of one of the conditions on $\PP$. Coreflectivity of $\BB$ in $\CC$ always allows for the provision of such a class $\PP$, and another strengthening of one of the axioms actually characterizes coreflectivity. Interestingly, the conditions on $\PP$ in their basic form are equivalent to a weakening of the concept of adjoint functors, introduced by Medvedew \cite{medvedev} under the name {\em semi-adjunction}; it emerges when one drops from the notion of adjunction one of the so-called triangular identities.

The paper is organized as follows. Section 2 summarizes all notions and terms that are needed for our exposition. Section 3 gives the categorical extension result, first in a basic and then in an advanced form, as indicated in the previous paragraph. 
In Section 4 we give a review of mostly known facts concerning the de Vries and the Fedorchuk dualities, augmented by some novel results, such as 
the generalization of  Bereznitskij's result on quasi-open maps, as mentioned above (see Proposition \ref{begen}).
After that, in Section 5,
applying our general categorical construction for the extension of dualities,
we present
an alternative proof of the Fedorchuk Duality Theorem, {\em without} making use of the
de Vries Duality Theorem. Thanks to this approach, we
 obtain a topological interpretation of all algebraic notions
used in the Fedorchuk Duality Theorem.
We start by proving that the category $\bf{ECH}_{\rm{op}}$ is coreflective in the category $\HC_{\rm{qop}}$ (see Proposition \ref{n1n}),
and we proceed by presenting a new direct proof of the Bezhanishvili theorem \cite[Theorem 8.1(1)]{Bezh} cited above  whose original proof uses
de Vries' Duality Theorem
(see Lemmas \ref{lem1} and \ref{lem2}, Proposition \ref{mainpro} and Corollary \ref{cor1f}). In this way the topological nature of the objects of the category $\Fed$ (and thus, of the objects of the category $\DHC$) gets transparently exposed.
For the morphisms of the category $\Fed$, Proposition \ref{morphpro} does the corresponding job. We are then in a position to
show that the category $\Fed$ is equivalent to the category $\DD$ as given by the categorical Proposition \ref{mere duality} in this concrete situation (see Theorem \ref{thnewfed}), completing in such a way our new proof of the Fedorchuk Duality Theorem.
Further, Proposition \ref{prff} clarifies the definition of the Fedorchuk dual equivalence $\Phi^a:\Fed\lra\HC_{{\rm qop}}$.
Finally, with the help of our categorical Theorem  \ref{what about S} for the extension of dualities, we reveal the connection between the Stone dual equivalence $S^t:\Stone\lra\Bool$ and the Fedorchuk dual equivalence $\Phi^t:\HC_{{\rm qop}}\lra\Fed$.

Our general references for unexplained notation are \cite{AHS} for category theory,
 \cite{E} for topology, and \cite{kop89} for Boolean algebras.

 \section{Preliminaries}

Below we first recall the notions of  {\em contact algebra}\/ and  {\em  normal contact algebra}.
They can be regarded as  algebraic analogues of proximity spaces (see \cite{EF,Sm,CE,AP,NW} for proximity spaces).
Generally speaking, in this paper we work mainly with Boolean algebras with  supplementary structures on them. In all cases, we will say that the
structured Boolean algebra in question is {\em complete}\/ if the underlying Boolean algebra is complete.
Our standard notation for the operations of a Boolean algebra $B$ is indicated by $B=(B, \land, \lor, {}^*, 0, 1)$; note in particular that the complement in $B$ is denoted by ${}^*$, and that $0$ and $1$ denote the least and largest element in $B$, not excluding the case $0=1$.

\begin{defi}\label{conalg}{\rm (\cite{DV1})}
\rm
A {\em Boolean contact algebra}, or, simply, {\em contact algebra}\/ (abbreviated as CA), is a structure $(B,C)$, where $B$ is a Boolean algebra, and $C$ a binary relation on $B$, called a {\em contact relation}, which satisfies the following axioms:

\begin{enumerate}
\renewcommand{\theenumi}{\ensuremath{({\rm C}\arabic{enumi})}}
 \item If $a\not= 0$ then $aCa$.\label{c1}
    \item If $aCb$ then $a\not=0$ and $b\not=0$.\label{c2}
     \item $aCb$ implies $bCa$.\label{c3}
     \item $aC(b\vee c)$  \tiff  $aCb$ or $aCc$. \label{c4}
\end{enumerate}

\noindent
 Two contact algebras\/ $(B,C)$ and\/
$(B',C')$ are said to be\/ {\em isomorphic}
if there exists a {\em CA-isomorphism} between them, {\em{i.e.}}, a Boolean
isomorphism $\varphi: B\longrightarrow B'$ such that, for all
$a,b\in B$, $aCb$ if and only if $\varphi(a)C' \varphi(b)$.

With  $-C$ denoting the set complement of $C$ in $B \times B$,
we shall consider two more properties of contact algebras:

\begin{enumerate}
\renewcommand{\theenumi}{\ensuremath{({\rm C}\arabic{enumi})}}
\setcounter{enumi}{4}
    \item If $a(-C)b$ then $a(-C)c$ and $b(-C)c^*$ for some
$c\in B$.\label{c5}
    \item  If $a\not= 1$ then there exists $b\not= 0$ such that
$b(-C)a$.  \label{c6}
\end{enumerate}
A contact algebra $(B,C)$ is called a {\it  Boolean normal contact  algebra}\/ or, briefly, a {\it  normal contact algebra}
(abbreviated as NCA) \cite{deV,F} if it satisfies \ref{c5} and \ref{c6}.
(Note that if $0\neq 1$, then \ref{c2} follows from the axioms \ref{c4}, \ref{c3}, and \ref{c6}.)
\end{defi}
The notion of normal contact algebra was introduced by Fedorchuk \cite{F} under the name {\em Boolean $\d$-algebra}\/, as an equivalent
expression of the notion of {\em compingent Boolean algebra} by de Vries (see the definition below). We call such algebras ``normal contact algebras'' because they form a subclass of the class of contact algebras which naturally arise in the context of
normal Hausdorff spaces (see \cite{DV1}).

\begin{defi}\label{def:ll}
\rm
For a contact algebra $(B,C)$ we define a binary relation  $\ll_C $ on $B$, called {\em non-tangential inclusion},  by
\begin{gather*}
\ a \ll_C b \text{ \tiff} a(-C)b^*.
\end{gather*}
%
If $C$ is understood, we shall  simply write
$\ll$ instead of $\ll_C$.
\end{defi}

The relations $C$ and $\ll$ are inter-definable. For example,
normal contact algebras may be defined equivalently -- and exactly
in this way they were introduced under the name of {\em
compingent Boolean algebras}\/ by de Vries in \cite{deV} -- as a pair consisting
of a Boolean algebra $B$ and a binary relation $\ll$ on $B$ satisfying the following axioms:

\begin{enumerate}
\renewcommand{\theenumi}{($\ll$\arabic{enumi})}
\item $a \ll b \timplies a \leq b$. \label{di1}
\item $0\ll 0$. \label{di2}
\item  $a\leq b\ll c\leq t$ implies $a\ll t$. \label{di3}
\item $a\ll c$ and $b\ll c$ implies $a\vee b\ll c$. \label{di4}
\item If  $a\ll c$ then $a\ll b\ll c$  for some $b\in B$. \label{di5}
\item  If $a\neq 0$ then there exists $b\neq 0$ such that $b\ll
a$. \label{di6}
\item$a\ll b$ implies $b^*\ll a^*$. \label{di7}
\end{enumerate}

Indeed, if $(B,C)$ is an NCA, then the relation $\ll_C$ satisfies the axioms \ref{di1} -- \ref{di7}.  Conversely,
having
 a pair $(B, \ll)$, where $B$ is a Boolean algebra and $\ll$ is a binary relation  on $B$ which satisfies \ref{di1} -- \ref{di7}, we define a relation
$C_\ll$
by $aC_\ll b$ \tiff $a\not\ll b^*$; then $(B,C_\ll)$ is an NCA. Note that the axioms \ref{c5} and \ref{c6} correspond to \ref{di5} and \ref{di6}, respectively. It is easy to see that a contact algebra could be equivalently defined as a
pair consisting of a Boolean algebra $B$ and a binary relation $\ll$ on $B$
subject to the  axioms \ref{di1} -- \ref{di4} and \ref{di7}.

The most important example of a CA is given by the regular closed sets of an arbitrary topological space $X$. Let us start with some standard notations and conventions that we use throughout the paper.
For a subset $M$ of $X$, we
denote by $\cl_X(M)$ (or simply
$\cl(M)$) the closure of $M$ in $X$, and by
 $\int(M)$ its
interior.
$\CO(X)$ denotes the set of
all clopen (= closed and open) subsets of $X$; trivially, $(\CO(X),\cup,\cap,\stm,\ems, X)$ is a Boolean algebra.
$\RC(X)$ (resp., $\RO(X)$) denotes the set of all
regular closed (resp., regular open) subsets of $X$; recall that a subset $F$ of $X$ is
said to be {\em regular closed}\/ (resp., {\em regular open}\/) if $F=\cl(\int(F))$ (resp., $F=\int(\cl(F))$)).

Note that in this paper (unlike in \cite{E}) compact spaces are not assumed to be Hausdorff.

\begin{exa}\label{rct}
\rm For a  topological space $X$, the collection $\RC(X)$  becomes a complete Boolean algebra
under the operations
\begin{align*}
F\vee G &\df F\cup G, & F\we G &\df \cl(\int(F\cap G)), & F^* &\df  \cl(X\stm F), & 0 &\df  \emptyset, & 1 &\df  X.
\end{align*}
 The infinite operations are given by the  formulas
 \begin{align*}
 \bigvee\{F_\g\st \g\in\GA\} &=  \cl(\bigcup_{\g\in\GA}F_\g)\
(=\cl(\bigcup_{\g\in\GA}\int(F_\g))=\cl(\int(\bigcup_{\g\in\GA}F_\g))), \\
\bigwedge\{F_\g\st \g\in\GA\} &=  \cl(\int(\bigcap\{F_\g\st
\g\in\GA\})).
 \end{align*}

\noindent One defines the  relation $\rho_X$ on $\RC(X)$ by setting, for each $F,G\in \RC(X)$,
 \begin{gather*}
 F \rho_XG \mbox{  \tiff  } F\cap G\neq \ems.
 \end{gather*}
 Clearly, $\rho_X$ is a contact relation on $\RC(X)$, called the {\em standard contact relation of $X$}. The  complete CA $(\RC(X),\rho_X)$  is called a {\em standard contact algebra}.
 Note that, for $F,G\in \RC(X)$,
 \begin{gather*}
 F\ll_{\rho_X}G \mbox{  \tiff  } F\sbe\int_X(G).
 \end{gather*}
Thus, if $X$ is a normal Hausdorff space then the standard contact algebra $(\RC(X),\rho_X)$ is a complete NCA.
\end{exa}

\begin{exa}\label{extrcr}
\rm Let $B$ be a Boolean algebra. Then there exist a largest and a
smallest contact relation on $B$; the largest one, $\rho_l$, is
defined by $$a\rho_l b \iff (a\neq 0\mbox{ and }b\neq 0),$$ and the
smallest one, $\rho_s$, by $$a\rho_s b \iff a\wedge b\neq 0.$$

\noindent Note that, for $a,b\in B$, $$a\ll_{\rho_s} b \iff a\le b;$$ hence
$a\ll_{\rho_s} a$, for any $a\in B$. Thus $(B,\rho_s)$ is a normal
contact algebra.
\end{exa}

We will need the following definition and assertion from \cite{DV1}:

\begin{defi}\label{relr}{\rm (\cite{DV1})}
\rm
For a contact algebra $(B,C)$ one defines the relation $R_{(B,C)}$ on the set of all filters on $B$ by
\begin{equation}\label{rcrel}
f\, R_{(B,C)}\, g$ \tiff $f\times g\sbe C,
\end{equation}
for all filters $f,g$ on $B$.
\end{defi}

\begin{pro}\label{relrult}
{\rm{(a)}} {\rm (\cite[Lemma 3.5, p. 222]{DV1})} \/ Let $(B,C)$ be a contact algebra. Then, for all $a,b\in B$, one has $a C b$ \tiff there exist ultrafilters $u,v$ in $B$ such that $a\in u$, $b\in v$ and $uR_{(B,C)} v$.

\medskip

\noindent{\rm{(b)}} {\rm (\cite{DV1,DUV})} \/ If $(B,C)$ is a normal contact algebra, then $R_{(B,C)}$ is an equivalence relation.
\end{pro}

\begin{defi}\label{defcluclan}
\rm  For CA $(B,C)$, a non--empty subset $\s$ of $B$ is called a \emph{cluster} if for all $x,y \in B$,
%
\begin{enumerate}
\renewcommand{\theenumi}{(CL\arabic{enumi})}
\item\label{cl1}  If $x,y\in\s$ then $xCy$;.
\item\label{cl2} If $x\vee y\in\s$ then $x\in\s$ or
$y\in\s$.
\item\label{cl3} If $xCy$ for every $y\in\s$, then $x\in\s$.
\end{enumerate}

\noindent The set of all clusters in an NCA $(B,C)$ is denoted by $\CL(B,C)$
\end{defi}

The next theorem is used  later on and may be proved exactly as Theorem 5.8 of \cite{NW}:

\begin{theorem}\label{conclustth}
A subset $\s$ of a normal contact algebra $(B,C)$ is a cluster \tiff
there exists an ultrafilter $u$ in $B$ such that
\begin{equation}\label{ultclu}
\ \s=\{a\in B\st aCb \mbox{ for every } b\in u\}.
\end{equation}
Moreover, given $\s$ and $a_0\in \s$, there exists an ultrafilter
$u$ in $B$ satisfying {\em (\ref{ultclu})} and containing $a_0$.
\end{theorem}

\begin{cor}\label{uniqult}
Let $(B,C)$ be a normal contact algebra and $u$ be an ultrafilter
 in $B$. Then there exists a unique
cluster $\s_u$ in $(B,C)$ containing $u$, namely
\begin{equation}\label{sigmau}
\ \s_u=\{a\in B\st aCb \mbox{  for every } b\in u\}.
\end{equation}
\end{cor}

The following simple result can be proved exactly as
Lemma 5.6 of \cite{NW}:

 \begin{fact}\label{fact29}
 Let $(B,C)$ be a normal contact algebra and $\s_1$, $\s_2$
 clusters in $(B,C)$. If  $\s_1\sbe \s_2$, then $\s_1=\s_2$.
 \end{fact}

\begin{nota}\label{sigmax}
\rm
 For a topological space $(X,\tau)$ and $x\in X$, we
set
\begin{equation}\label{sxvx}
\  \s_x^X=\{F\in RC(X)\st x\in F\}
\end{equation}
and often write just $\s_x$.
\end{nota}

The next assertion is obvious:

\begin{fact}\label{sxcluster}
For a regular topological space $X$,  $\s_x$ is a
cluster in the CA $(RC(X),\rho_X)$, called a {\em point-cluster}.
\end{fact}

For a category $\CC$, we denote by $\card\CC$ its class of objects, by $\Mor(\CC)$ its class of morphisms, and by $\CC(X,Y)$ the set of all
   $\CC$-morphisms $X\lra Y$.

\begin{nist}\label{Std}
\rm
 Let us fix the notation for the Stone Duality (\cite{ST,kop89}).  We denote by $\Stone$ the category of all zero-dimensional compact Hausdorff  spaces (= {\em Stone spaces}) and their continuous mappings, and by $\Bool$ the category of  Boolean algebras and Boolean  homomorphisms.
 The contravariant functors
furnishing the Stone duality
are denoted by $$S^a:{\Bool}\lra{\Stone}\ \  \mbox{ and }\ \  S^t:\Stone\lra\Bool.$$
Hence, for $A\in\card{\Bool}$,
$S^a(A)$  is the set  $\ult (A)$ of all
ultrafilters in $A$ endowed with the topology whose open base is
the family
$\{s_A(a)\st a\in A\}$, where $$s_A(a)\df\{u\in \ult (A)\st a\in u\}$$
for all $a\in A$.
For $X\in\card{\Stone}$, one sets $S^t(X)\df \CO(X),$ and
for morphisms $f\in\Stone(X,Y)$ and $\p\in\Bool(B_1,B_2)$
one puts
$$S^t(f)(F)=f\inv(F)\text{ and }S^a(\p)(u)=\p\inv(u)
$$
 for all $F\in \CO(Y)$ and $u\in \ult (B_2)$.
Now, for every Boolean algebra $A$, the map
$$s_A:A\lra S^t(S^a(A)), \ a\mapsto s_A(a),$$
is a Boolean
isomorphism, and for every Stone space $X$,
the map
$$t_X:X\lra S^a(\CO(X)), \ \ x\mapsto u_x,$$
is a homeomorphism; here,
for every $x\in X$,

\begin{equation}\label{ux}
u_x\df\{P\in {\rm{CO}}(X)\st x\in P\}.
\end{equation}
Moreover, $s_A$ and $t_X$ are natural in $A$ and $X$.
\end{nist}

\begin{nist}\label{niab}
\rm

Let us recall some standard properties for a continuous map of topological spaces: $f:X\lra Y$ is

\begin{itemize}

\item {\em closed}\/ if the image of each closed set is closed;

\item {\em open}\/ if the image of each open set is open;

\item {\em perfect}\/ if it is closed and has compact fibres; 

\item {\em
quasi-open\/} (\cite{MP}) if $\int(f(U))\nes$ for every non-empty open subset $U$
of $X$;


\item {\em skeletal}\/ (\cite{MR})
if for every open subset  $V$  of $Y$%
\begin{equation}\label{ske}
\ \int(f\inv(\cl (V)))\sbe\cl(f\inv(V));
\end{equation}

\item
{\em irreducible}\/ if $f(X)=Y$ and if, for every proper closed
subset $F$ of $X$, $f(F)\neq Y$.
\end{itemize}

 Recall  that, for a regular space $X$, a space $EX$ is called an
{\em absolute of} $X$ if there exists a perfect irreducible mapping $\pi_X : EX\lra X$
and every perfect irreducible preimage of $EX$ is homeomorphic to $EX$ (see,
{\em e.g.}, \cite{ArP,PS}).  It is well-known that:

\medskip

\noindent(a) the absolute is unique up to
homeomorphism;

\medskip

\noindent(b) a space $Y$ is an absolute of a regular space $X$ \tiff $Y$ is
an extremally disconnected Tychonoff space for which there exists a perfect
irreducible mapping $\pi: Y\lra X$; such mappings $\pi$ are called {\em projective covers of} $X$;

\medskip

\noindent(c) if $X$ is a compact Hausdorff space, then it is well-known that
$EX = S^a(\RC(X))$  and the projective cover $\pi_X$ of $X$ is defined by $$\pi_X(u)\df \bigcap u,$$ for every $u\in \ult (\RC(X))\ (=\ S^a(\RC(X))$ (here $S^a:\Bool\lra \Stone$ is the Stone contravariant functor).
\end{nist}

\begin{nist}\label{nipr}
{\rm
Let $\CC$ be a subcategory of the category $\Top$ of all topological spaces and all continuous mappings between them.
Recall that a $\CC$-object $P$ is called a {\em projective object}\/ in $\CC$ if for every $g\in\CC(P,Y)$ and every
perfect surjection $f\in\CC(X,Y)$, there exists  $h\in\CC(P,X)$ such that $f\circ h=g$.

A. M. Gleason \cite{Gle} proved:}

\medskip

In the category $\HC$ of compact Hausdorff spaces and continuous mappings, the projective objects are precisely the extremally disconnected spaces.
\end{nist}

\section{Extensions of dualities}
\begin{nist}\label{intro}
\rm
Given a dual equivalence $T:\AA\lra\BB$ and an embedding $I$ of $\BB$ as a full subcategory of a category $\CC$, we wish to give a {\em natural construction} for a category $\DD$ into which $\AA$ may be fully embedded via $J$, such that $T$ extends to a dual equivalence $\tilde{T}:\DD\lra\CC$:

\begin{center}
$\xymatrix{\DD\ar[r]^{\tilde{T}} & \CC\\
            \AA\ar[r]_{T}\ar[u]^{J} & \BB\ar[u]_{I}}$
            \end{center}
Our construction depends on a class $\PP$ of morphisms of $\CC$ satisfying certain conditions, which are closely related to certain properties of the full embedding $I$. It turns out that, when $\BB$ is coreflective in $\CC$, such a class $\PP$ always exists.

We call a class $\PP$ of morphisms in $\CC$ a $(\BB,\CC)$-{\em covering class} if it satisfies the following conditions (P1-5):
\medskip

\noindent{\rm  (P1)} $\forall \;(p\!:B\lra C)\; \in\,\PP:\;B\in|\BB|$;

\medskip

\noindent{\rm  (P2)} $\forall B\in|\BB|:1_B\in\PP$;

\medskip

\noindent{\rm  (P3)} $\PP\circ {\rm Iso}(\BB)\sbe\PP$;

\medskip

\noindent{\rm (P4)} $\forall C\in |\CC|\;\,\exists\; (p:B\lra C)\in \PP$;

\medskip

\noindent{\rm (P5)} {\em for morphisms in $\CC$, there is a functorial assignment}

\begin{center}
$\quad\quad\quad\xymatrix{B\ar[d]_{p} & B'\ar[d]^{p'}\\
            C\ar[r]_{v} & C'}$
            \hfil\quad\quad\quad
           $ \mapsto$
            \hfil
$\xymatrix{B\ar[r]^{\hat{v}}\ar[d]_{p} & B'\ar[d]^{p'}\\
            C\ar[r]_{v} & C'}$
            \end{center}
$$((p\!:B\!\to\! C)\!\in\!\PP,v\!:C\to C',\,(p'\!:B'\!\to\! C')\!\in\!\PP) \mapsto\ (\hat{v}\!:B\to B'\;with\;v\circ p=p'\circ \hat{v}),$$
\noindent {\em so that}
\begin{center}
$\quad\quad\xymatrix{B\ar[d]_{p} & B'\ar[d]^{p'} & B''\ar[d]^{p''}\\
            C\ar[r]_{v} & C'\ar[r]_{w} & C''}$
            \quad\quad\quad\quad\quad
           $ \mapsto$
            \hfil
$\xymatrix{B\ar[r]_{=\hat{w}\circ \hat{v}}^{\widehat{w\circ v}}\ar[d]_{p} & B''\ar[d]^{p''}\\
            C\ar[r]_{w\circ v} & C''}$
            \end{center}
            \noindent {\em and}
\begin{center}
$\quad\quad\quad\xymatrix{B\ar[d]_{p} & B\ar[d]^{p}\\
            C\ar[r]_{1_{C}} & C}$
            \hfil\quad\quad\quad
           $ \mapsto$
            \hfil
$\xymatrix{B\ar[r]^{\widehat{1_{C}}}_{=1_{B}}\ar[d]_{p} & B\ar[d]^{p}\\
            C\ar[r]_{1_{C}} & C.}$
            \end{center}
\noindent We emphasize  that, in the given assignment, $\hat{v}$ depends not only on $v$, but also on $p$ and $p'$.
Next, we note that, in the presence of (P3), condition (P2) means equivalently

\medskip

\noindent (P2$'$) ${\rm Iso}(\BB)\sbe\PP$.

\medskip

\noindent In condition (P4) we tacitly assume that, for every  $C\in|\CC|$, we have a {\em chosen} morphism $p\in\PP$ with codomain $C$. In the presence of (P2), that morphism may be taken to be an identity morphism whenever $C\in|\BB|$. To emphasize the choice, we may reformulate (P4), as follows:

\medskip

\noindent{\rm (P4$'$)} $\forall C\in |\CC|\;\,\exists\; (\pi_C:EC\lra C)\in \PP$ {\em (with $\pi_C=1_C$ when $C\in |\BB|$)}.

\medskip

\noindent It is now clear that (P5) enables us to make $E$ a functor $\CC\lra\BB$ and $\pi$ a natural transformation $IE\lra{\sf Id}_{\CC}$: for $v:C\lra C'$ in $\CC$, one simply considers the commutative diagram
\begin{center}
$\xymatrix{EC\ar[r]^{Ev\df\hat{v}}\ar[d]_{\pi_C} & EC'\ar[d]^{\pi_{C'}.}\\
            C\ar[r]_{v} & C'.}$
            \end{center}
Then $\pi I$ is an isomorphism, and even the identity transformation $1_I$ if we choose $\pi_B=1_B$ for all $B\in|\BB|$, in which case $EI={\sf Id}_{\BB}$. This proves the ``only if\,"-part of the following proposition; to prove its ``if\,"-part, we need to use again a large choice principle, as follows.

 \medskip

\noindent{\bf Proposition.}\label{functorial} {\em The full subcategory $\BB$ of $\CC$ with inclusion functor $I$ admits a $(\BB,\CC)$-covering class if, and only if, there are a functor $E:\CC\lra\BB$ and a natural transformation $\pi:IE\lra{\sf Id}_{\CC}$, such that $\pi I:IEI\lra I$ is an isomorphism; $E$ and $\pi$ may actually be chosen to satisfy $EI={\sf Id}_{\BB}$ and $\pi I=1_I$.}

\medskip
\doc
For the ``if\,"-part of the assertion, we consider the morphism class
$$\PP_{\pi}\df\{p:B\to C \,{\rm in}\, \CC\,|\,B\in|\BB|,\,p=\pi_C\circ \beta \;{\rm for\;some\; isomorphism\,}\beta:B\lra EC\}.$$
By definition, the domain of every morphism in $\PP_{\pi}$ lies in $\BB$, and pre-composition with any $\BB$-isomorphism gives again a morphism in $\PP_{\pi}$. Since $\pi_B$ with $B\in|\BB|$ is an isomorphism, $1_B\in\PP_{\pi}$ follows. This shows (P1-3), and also (P4) holds trivially. To show (P5), for every $p:B\lra C$ in $\PP_{\pi}$, we choose an isomorphism $\beta_p: B\lra EC$ with $\pi_C\circ\beta_p=p$. Then, for every $v:C\lra C'$ in $\CC$ with $p:B\lra C,\;p':B'\lra C'$ in $\PP$, we put
$$ \hat{v}\df\beta_{p'}^{-1}\circ Ev\circ\beta_p:B\lra B'.$$
The functoriality of this assignment follows from the functoriality of $E$, and the naturality of $\pi$ gives
$p'\circ\hat{v}=p'\circ\beta_{p'}^{-1}\circ Ev\circ\beta_p=\pi_{C'}\circ Ev\circ \beta_p=v\circ\pi_C\circ \beta_p=v\circ p.$
\sqs
\end{nist}

\begin{nist}\label{stronger conditions}
\rm

We note that, starting with a $(\BB,\CC)$-covering class $\PP$, if we first construct $E,\pi$ and then the class $\PP_{\pi}$ as in the proof of Proposition \ref{functorial}, we get back the class $\PP$. Indeed, trivially $\PP_{\pi}\subseteq\PP$; conversely, given $p\in \PP$, one uses the following diagram to show $p=\pi_C\circ\beta$ for an isomorphism $\beta$, so that $p\in\PP_{\pi}$:
\begin{center}
$\xymatrix{B\ar[r]^{\beta\df\widehat{1_C}}\ar[d]_{p} & EC\ar[d]^{\pi_{C}}\\
            C\ar[r]_{1_C} & C.}$
            \end{center}

\noindent Later on, we need to ensure the uniqueness of such an isomorphism $\beta$. For that, it clearly suffices that the morphism $\pi_C$ is {\em rigid}, in the sense that $\pi_C\circ\alpha=\pi_C$ for an isomorphism $\alpha:EC\lra EC$ is possible only for $\alpha=1_{EC}.$ We therefore consider the
following strengthening of condition (P4) or, equivalently, of (P4$'$):

\medskip

\noindent{\rm (P4$^*$)} $\forall C\in |\CC|\;\,\exists\; (\pi_C:EC\lra C)\in \PP,\,\pi_C$ {\em rigid (with $\pi_C=1_C$ when $C\in |\BB|$)}.

\medskip


\noindent We call the $(\BB,\CC)$-covering class $\PP$ {\em rigid} if it satisfies (P4*) (instead of just (P4)).
The natural transformation $\pi:IE\lra{\sf Id}_{\CC}$ as in Proposition \ref{functorial} (so that $EI={\sf Id}_{\BB}$ and $\pi I =1_I$) is called {\em rigid} if every morphism $\pi_C\;(C\in|\CC|)$ is rigid. The dual notion is {\em corigid}.

In the presence of (P1-4), the following condition obviously  implies both, (P4$^*$) and (P5):
\medskip

\noindent {\em ${\rm (P5^*)}$ For all $v:C\lra C'$ in $\CC$ and $p:B\lra C,\; p':B'\lra C'$ in $\PP$, there is precisely one morphism $\hat{v}:B\lra B'$ with $v\circ p=p'\circ\hat{v}$.}

\medskip
 \noindent When $\BB$ is coreflective in $\CC$, the functor $E$ in Proposition \ref{functorial} may be chosen to be the coreflector and $\pi$ the coreflection. Actually, it is obvious that the couniversal property of the coreflections is, in the presence of (P1-4), equivalently expressed by (P5$^*$). In summary, we obtain with Proposition \ref{functorial}:

\medskip

\noindent{\bf Corollary.}\label{coreflective} {\em The full subcategory $\BB$ of $\CC$ admits a pair $E,\pi$ as in {\em Proposition \ref{functorial}} with $\pi$ rigid if, and only if, it admits a rigid $(\BB,\CC)$-covering class. $\BB$ is coreflective in $\CC$ if, and only if, there is a class $\PP$ of morphisms in $\CC$ satisfying properties {\em (P1-4)} and {\em (P5$^*$)}, making $\PP$ in particular a rigid $(\BB,\CC)$-covering class.}

\end{nist}

\begin{nist}\label{construction of D}
\rm
In addition to the full subcategory $\BB$ of $\CC$ with inclusion functor $I$ and a $(\BB,\CC)$-covering class $\PP$ ({\em not} necessarily rigid), so that (according to Proposition \ref{functorial}) there are a functor $E:\CC\lra\BB$ and a natural transformation $\pi:IE\lra {\sf Id}_{\CC}$, with $\pi_B$ an isomorphism for all $B\in |\BB|$,
 let us now consider a dual equivalence $(S,T,\eta,\ep)$ with contravariant functors
 $$T:\AA\lra\BB\quad{\rm and}\quad S:\BB\lra \AA$$
and natural isomorphisms $\eta:{\rm Id}_{\BB}\lra T\circ S\ \mbox{ and }\ \ep:{\rm Id}_{\AA}\lra S\circ T$ which, without loss of generality, may be assumed to satisfy the {\em triangular identities}
$$T\varepsilon_A\circ \eta_{TA}=1_{TA}\quad{\rm and}\quad S\eta_B\circ\varepsilon_{SB}=1_{SB},$$
for all $A\in|\AA|, B\in |\BB|$. We construct the category $\DD$ as envisaged at the beginning of \ref{intro}, as follows:

\begin{itemize}

\item objects in $\DD$ are pairs $(A,p)$ with $A\in|\AA|$ and $p:TA\lra C$ in the class $\PP$;

\item  morphisms $(\p,f):(A,p)\lra(A',p')$ in $\DD$ are given by
morphisms $\p:A\lra A'$ in $\AA$ and  $f:C'\lra C$ in $\CC$,
such that, in the notation of (P5), $T\varphi=\hat{f}$:
\begin{center}
$\xymatrix{TA\ar[d]_{p} & TA'\ar[l]_{T\varphi=\hat{f}}\ar[d]^{p'}\\
            C & C'\ar[l]^{f}}$
            \end{center}

\item    composition is as in $\AA$ and $\CC$; that is, $(\p,f)$ as above gets composed with
$(\p',f'):(A',p')\lra(A'',p'')$ by the horizontal pasting of diagrams, that is,
$$(\p',f')\circ(\p,f)\df (\p'\circ\p,f\circ f').$$

\item the identity morphism of a $\DD$-object $(A,p)$ is the $\DD$-morphism $(1_A,1_{\cod(p)})$.
\end{itemize}

\noindent Of course, the fact that the composition and the identity morphisms of $\DD$ are well defined, relies heavily on (P5). Since $T$ is fully faithful, we note that, for a morphism $(\varphi,f)$ in $\DD$, the $\AA$-morphism $\varphi$ is determined by $\hat{f}$ and, hence, by $f,p,$ and $p'$. As a particular consequence, by (P2), we have a full embedding $J:\AA\lra\DD$, defined by
$$(\p:A\lra A')\mapsto (\;J\p\df(\p,T\p):(A, 1_{TA})\lra(A',1_{TA'})\;);$$
we note that, here,  $T\varphi=\widehat{T\varphi}$ holds trivially. We also note that $J$ has a retraction: the functor $F:\DD\lra\AA$ with $F(\varphi,f)=\varphi$ trivially satisfies $FJ={\sf Id}_{\AA}$.

It is now straightforward to establish a dual equivalence of $\DD$ with $\CC$, as follows:

\medskip

\noindent{\bf Proposition.}\label{mere duality} {\em There is a dual equivalence $\tilde{T}:\DD\lra\CC$ extending the given dual equivalence $T:\AA\lra\BB$, so that $\tilde{T}J=IT$; furthermore, with $E$ as in {\em Proposition \ref{functorial}}, one has $E\tilde{T}\cong TF$:}

\begin{center}
$\xymatrix{\DD\ar[r]^{\tilde{T}} & \CC\\
            \AA\ar[r]_{T}\ar[u]^{J} & \BB\ar[u]_{I}}$
            \hfil
$\xymatrix{\DD\ar[d]_{F}\ar[r]^{\tilde{T}}\ar@{}[rd]|{\cong} & \CC\ar[d]^{E}\\
            \AA\ar[r]_{T} & \BB}$
            \end{center}

\medskip

\doc
The contravariant functor $\tilde T:\DD\to\CC$ is simply the projection
$$(\;(\p,f):(A,p)\lra(A',p')\;)\;\mapsto\;(C\lla C':f).$$
As noted above, since $T$ is fully faithful, $\varphi$ is determined by $f, p, p'$ and the condition $T\varphi=\hat{f}$, which makes $\tilde{T}$ fully faithful as well. Moreover, given $C\in|\CC|$, since $T$ is essentially surjective on objects, one has an isomorphism $\alpha:TA\lra EC$ with some $A\in|\AA|$ and, hence, $\tilde{T}(A, \pi_C\circ\alpha)=C$, so that $\tilde{T}$ is actually surjective on objects.

The identity $\tilde{T}J=IT$ holds trivially. We construct a natural isomorphism $\beta:TF\lra E\tilde{T}$, as follows: for $(A,p)\in|\DD|$ with $p:TA\lra C$ in $\PP$, we define $\beta_{(A,p)}: TF(A,p)=TA\lra EC=E\tilde{T}(A,p)$, written shortly as $\beta_p$, by the diagram
\begin{center}
$\xymatrix{TA\ar[r]^{\beta_p\df\widehat{1_C}}\ar[d]_{p} & EC\ar[d]^{\pi_{C}}\\
            C\ar[r]_{1_C} & C.}$
            \end{center}
Functoriality of $(\widehat{-})$ shows that, with $1_C$, also $\beta_p$ is an isomorphism. Since on morphisms also $E$ is defined via $(\widehat{-})$ (see Proposition \ref{functorial}), one exploits its functoriality again to see that the upper square of
\begin{center}
$\xymatrix{TA\ar[d]^{\beta_p}\ar@/_2pc/[dd]_{p} & {TA'}\ar[l]_{T\p}\ar[d]_{\beta_{p'}}\ar@/^2pc/[dd]^{p'}\\
            {EC}\ar[d]^{\pi_C} & {EC'}\ar[l]_{Ef}\ar[d]_{\pi_{C'}}\\
            C & {C'}\ar[l]_{f}}$
            \end{center}
commutes for every morphism $(\varphi,f):(A,p)\lra(A',p')$ in $\DD$, which shows the naturality of $\beta$.
\sqs
\end{nist}

\begin{nist}\label{what about S}
\rm
The virtue of any dual equivalence depends on one's ability to go back and forth efficiently between the dually equivalent categories. For that, it is important to not only ``lift" the functor $T$ to a dual equivalence $\tilde{T}$, but to do the same also for its partner $S$, as well as to the natural isomorphisms $\eta$ and $\varepsilon$ (as listed at the beginning of \ref{construction of D}).

The following theorem augments Proposition \ref{mere duality}, by laying out the functorial connections of the lifted dual equivalence with the given dual equivalence. To be able to do so to the fullest extent, we need to assume rigidity as defined in \ref{stronger conditions} and express the existence hypothesis of a {\em rigid} $(\BB,\CC)$-covering class functorially (see Proposition \ref{functorial} and Corollary \ref{stronger conditions}); any coreflective $\BB$ in $\CC$ admits such a class.

\medskip
\noindent{\bf Theorem} ({\em Extension Theorem}). {\em Let $\BB$ be a full subcategory of $\CC$ with inclusion functor $I:\BB\lra\CC$, such that

\medskip
\noindent {\em ($\ast$)} $I$ admits a retraction $E$ and a rigid natural transf. $\pi:IE\lra{\sf Id}_{\CC}$ with $\pi I=1_I$.
\medskip

\noindent Then, for every dual equivalence $T\!\!:\AA\longleftrightarrow\BB:\!\!S$, with natural isomorphisms
$\eta:{\rm Id}_{\BB}\lra T\circ S\ \mbox{ and }\ \ep:{\rm Id}_{\AA}\lra S\circ T$ satisfying the triangular identities, there are

\begin{itemize}
\item a category $\DD$ as described in {\em \ref{mere duality}}  (with $\PP=\PP_{\pi}$ of  {\em Proposition \ref{functorial}});
\item a full embedding $J:\AA\lra \DD$, such that
\end{itemize}

\noindent {\em ($\ast^{\rm op}$)} $J$ admits a retraction $F$ and a corigid natural tr. $\rho:{\sf Id}_{\DD}\lra JF$ with $\rho J=1_J$;

\begin{itemize}

\item a dual equivalence $\tilde{T}\!\!:\DD\longleftrightarrow\CC:\!\!\tilde{S}$, with natural isomorphisms $\tilde{\eta}:{\rm Id}_{\CC}\lra \tilde{T}\circ \tilde{S}\ \mbox{ and }\ \tilde{\ep}:{\rm Id}_{\DD}\lra \tilde{S}\circ \tilde{T}$ satisfying the triangular identities;
\item and natural isomorphisms $\beta:TF\lra E\tilde{T}$ and $\gamma:JS\lra\tilde{S}I.$
\end{itemize}
These data satisfy the following identities:
 \begin{itemize}
 \item[{\em (1)}] $\tilde{T}J=IT$ and $F\tilde{S}=SE$;
 \item[{\em (2)}] $\tilde{T}\tilde{S}={\sf Id}_{\CC},\,\tilde{\eta}=1_{{\sf Id}_{\CC}},$ and $\tilde{T}\tilde{\varepsilon}=1_{\tilde{T}},\,\tilde{\varepsilon}\tilde{S}=1_{\tilde{S}}$;
 \item[{\em(3)}] $\pi\tilde{T}\circ I\beta=\tilde{T}\rho$ and $\gamma E\circ\rho\tilde{S}=\tilde{S}\pi$;
 \item[{\em(4)}] $\tilde{T}\gamma=I\eta$ and $S\beta\circ F\tilde{\varepsilon}=\varepsilon F$.
 \end{itemize}
}
\medskip

\begin{proof}
For the definition of $\DD, J, F$ and $\beta$, we refer to Proposition \ref{construction of D}, which affirms all statements of the Theorem involving only these entities. Next, let us construct a natural transformation $\rho:{\sf Id}_{\DD}\lra JF$. For $(A,p:TA\lra C)\in|\DD|$ we define
$$\rho_{(A,p)}\df (\iota_p, p):(A,p)\lra(A,1_{TA})=JF(A,p).$$
with $\iota_p:A\to A$ in $\AA$ determined by the commutative diagram
\begin{center}
$\xymatrix{TA\ar[d]_{p} & TA\ar[l]_{\hat{p}=T\iota_p}\ar[d]^{1_{TA}}\\
            C & TA\ar[l]_{p}.}$
            \end{center}
For the naturality of $\rho$, consider a morphism $(\varphi, f):(A,p)\lra(A',p')$ in $\DD$ and form the commutative diagrams

\begin{center}
$\xymatrix{TA\ar[d]_{p} & TA\ar[l]_{\hat{p}=T\iota_p}\ar[d]^{1_{TA}} & TA'\ar[l]_{\widehat{T\varphi}=T\varphi}\ar[d]^{1_{TA'}}\\
            C & TA\ar[l]^{p} & TA'\ar[l]^{T\varphi}}$
            \hfil
$\xymatrix{TA\ar[d]_{p} & TA'\ar[l]_{\hat{f}=T\varphi}\ar[d]^{p'} & TA'\ar[l]_{\widehat{p'}=T\iota_{p'}}\ar[d]^{1_{TA'}}\\
            C & C'\ar[l]^{f} & TA'.\ar[l]^{p'}}$
 \end{center}
Since their bottom-row composites coincide, so do their top-row composites, and since $T$ is faithful, $\varphi\circ\iota_p=\iota_{p'}\circ\varphi$ follows. Hence,  $(\varphi, T\varphi)\circ(\iota_p, p)=(\iota_{p'},p')\circ(\varphi,f)$, which is precisely the required naturality condition $JF(\varphi,f)\circ\rho_{(A,p)}=\rho_{(A',p')}\circ(\varphi,f).$ Obviously, $\rho_{JA}=(1_A,1_{TA})=1_{JA}$
and $\tilde{T}\rho_{(A,p)}=p=\pi_C\circ\beta_p=\pi_{\tilde{T}(A,p)}\circ I\beta_{(A,p)}$ for all $(A,p:TA\to C)\in|\DD|$, so that $\rho J=1_J$ and $\tilde{T}\rho=\pi\tilde{T}\circ I\beta$ (see (3)).

We must also check that $\rho$ is corigid which, not surprisingly, requires us to use the rigidity of $\pi$ (for the first time in this proof). Indeed, the rigidity of $\pi$ makes every $p\in \PP_{\pi}$ rigid. Consequently, for every object $(A,p)$ in $\DD$ and every isomorphism $JF(A,p)\lra JF(A,p)$, which may be assumed to be of the form $J\alpha$ with an isomorphism $\alpha: A\lra A$ in $\AA$, such that $J\alpha\circ\rho_{(A,p)}=\rho_{(A,p)}$, we have $p\circ T\alpha=p$, so that $T\alpha=1_{TA}=T1_A$ and then $\alpha=1_A$ follows.

The adjoint $\tilde{S}:\CC\to\DD$ of $\tilde{T}$ is defined by
$$(C\lla C':f)\;\mapsto\;(\;\tilde{S}f\df(SEf,f):(SEC,\pi_C\circ\eta_{EC}^{-1})\lra(SEC',\pi_{C'}\circ\eta_{EC'}^{-1})\;).$$
That $\tilde{S}f$ is indeed a morphism in $\DD$ may be seen with the commutative diagram
\begin{center}
$\xymatrix{TSEC\ar[d]^{\eta_{EC}^{-1}}\ar@/_2pc/[dd]_{p} & {TSEC'}\ar[l]_{T(SEf)}\ar[d]_{\eta_{EC'}^{-1}}\ar@/^2pc/[dd]^{p'}\\
            {EC}\ar[d]^{\pi_C} & {EC'}\ar[l]_{Ef}\ar[d]_{\pi_{C'}}\\
            C & {C'}\ar[l]_{f}}$
            \end{center}
where $p\df\pi_C\circ\eta_{EC}^{-1}$,  $p'\df\pi_{C'}\circ\eta_{EC'}^{-1}$. Indeed, since $\beta_p=\eta_{EC}^{-1}$ by rigidity of $\pi_C$, one has

$$\hat{f}=\beta_p^{-1}\circ Ef\circ \beta_{p'}=\eta_{EC}\circ Ef\circ\eta_{EC'}^{-1}=TS(Ef)\circ \eta_{EC'}\circ\eta_{EC'}^{-1}=T(SEf).$$
The remaining identity $F\tilde{S}=SE$ of (1) holds trivially.

The natural transformation $\tilde{\eta}:{\sf Id}_{\CC}\lra\tilde{T}\tilde{S}$ is defined by

\medskip

$\tilde{\eta}_C\df 1_C: C\lra C=\tilde{T}(SEC,\pi_C\circ\eta_{EC}^{-1})=\tilde{T}\tilde{S}C$

\medskip
\noindent for all $C\in|\CC|$. We see that we actually have $\tilde{T}\tilde{S}={\sf Id}_{\CC}$ and $\tilde{\eta}=1_{{\sf Id}_{\CC}}.$ For defining the natural transformation $\tilde{\varepsilon}:{\sf Id}_{\DD}\lra\tilde{S}\tilde{T}$ we may put, for all $(A,p)\in|\DD|$ with $p:TA\to C$ in $\PP$,

\medskip

$\tilde{\varepsilon}_{(A,p)}\df (S\beta_p^{-1}\circ\varepsilon_A,1_C): (A,p)\lra(SEC,\pi_C\circ\eta_{EC}^{-1})=\tilde{S}C=\tilde{S}\tilde{T}(A,p).$

\medskip

\noindent Indeed, for $\p\df S\beta_p^{-1}\circ\varepsilon_A$, the diagram
\begin{center}
$\xymatrix{TA\ar[d]^{\beta_p}\ar@/_2pc/[dd]_{p} & {TSEC}\ar[l]_{T\p}\ar[d]_{\eta_{EC}^{-1}}\ar@/^2pc/[dd]^{\pi_C\circ \eta_{EC}^{-1}}\\
            {EC}\ar[d]^{\pi_C} & {EC}\ar[l]_{E1_C}\ar[d]_{\pi_{C}}\\
            C & {C}\ar[l]_{1_C}}$
            \end{center}
shows, with $\eta_{TA}^{-1}=T\varepsilon_A$, that one has
$$\widehat{1_C}=\beta_p^{-1}\circ E1_C\circ \eta_{EC}^{-1}=\eta_{TA}^{-1}\circ TS\beta_p^{-1}=T(S\beta_p^{-1}\circ \varepsilon_A)=T\p,$$
as desired. Clearly, $\tilde{\varepsilon}_{(A,p)}$ is an isomorphism, and $\tilde{T}\tilde{\varepsilon}_{(A,p)}=1_C$. Checking the naturality of $\tilde{\varepsilon}_{(A,p)}$ in $(A,p)$ is now a routine matter, which we may skip here. Consequently, $\tilde{T}\tilde{\varepsilon}=1_{\tilde{T}}$. For the remaining identity $\tilde{\varepsilon}\tilde{S}=1_{\tilde{S}}$ of (2), we compute, for all $C\in|\CC|$,
$$\tilde{\varepsilon}_{\tilde{S}C}=\tilde{\varepsilon}_{(SEC,\pi_C\circ\eta_{EC}^{-1})}=(S\beta_{\pi_C\circ\circ\eta_{EC}^{-1}}\circ\varepsilon_{SEC},1_C)=(S\eta_{EC}\circ\varepsilon_{SEC},1_C)=1_{\tilde{S}C}.$$
Note that, as a trivial consequence of the identities in (2), one has in particular the triangular identities
$$\tilde{T}\tilde{\varepsilon}\circ\tilde{\eta}\tilde{T}=1_{\tilde{T}}\quad {\rm and}\quad \tilde{S}\tilde{\eta}\circ\tilde{\varepsilon}\tilde{S}=1_{\tilde{S}}.$$

It remains to define a natural isomorphism $\gamma:JS\lra\tilde{S}I$: we put
$$\gamma_B\df(S\pi_B,\eta_B):JSB=(SB,1_{TSB})\lra(SEB,\pi_B\circ\eta_{EB}^{-1})=\tilde{S}B$$
for all $B\in|\BB|$.
By naturality of both, $\pi$ and $\eta$, the diagram
\begin{center}
$\xymatrix{TSB\ar[d]^{\pi_{TSB}^{-1}}\ar@/_2pc/[dd]_{1_{TSB}} & {TSEB}\ar[l]_{T(S\eta_B)}\ar[d]_{\eta_{EB}^{-1}}\ar@/^2pc/[dd]^{\pi_{B}\circ \eta_{EB}^{-1}}\\
            {ETSB}\ar[d]^{\pi_{TSB}} & {EB}\ar[l]_{E\eta_B}\ar[d]_{\pi_{B}}\\
            {TSB} & {B}\ar[l]_{\eta_B}}$
            \end{center}
commutes, so that $\gamma_B$ is indeed a morphism in $\DD$, and it is an isomorphism as both, $S\pi_B$ and $\eta_B$, are isomorphisms. The naturality check is straightforward. The verification of the remaining identities in (3) and (4) may be left to the reader.
\end{proof}
\end{nist}

\begin{nist}\label{semi}
\rm
We briefly return to Proposition \ref{functorial}, just to emphasize that the types of full embeddings characterized by it already appeared implicitly in early categorical work. Recall that an arbitrary functor $I:\BB\lra\CC$ is {\em left semi-adjoint} \cite{medvedev} if there are a functor $E:\CC\lra\BB$ and natural transformations $\pi:IE\lra{\sf Id}_{\CC}$ and $\sigma:{\sf Id}_{\BB}\lra EI$ satisfying the {\em triangular identity}
$$\pi I\circ I\sigma=1_I;$$
we call $I$ {\em fully left semi-adjoint}\/ if, in addition, $\sigma$ may be chosen to be an isomorphism.

Although they are missing the second triangular identity satisfied by a left-adjoint functor $I$ ({\em i.e.}, $E\pi\circ \sigma E=1_E$ will generally {\em not} hold), left semi-adjoints still enjoy the most important property of left-adjoint functors: they preserve all existing colimits. Indeed, the standard proof for left-adjoint functors still works for left semi-adjoint functors; see \cite{borger-tholen}.

Fully semi left-adjoint functors may be characterized as in Proposition \ref{functorial}:

\medskip

\noindent{\bf Lemma.}\label{lemma} {\em A functor $I:\BB\lra\CC$ is fully left semi-adjoint if, and only if, $I$ is full and faithful, and there are a functor $E:\CC\to\BB$ and a natural transformation $\pi:IE\lra{\sf Id}_{\CC}$ with $\pi I$ an isomorphism. Furthermore, $I$ is then left adjoint precisely when $IE\pi=\pi IE$.}

\medskip
\doc
For $I$ full and faithful, having $E$ and $\pi$ with $\pi I$ being an isomorphism, we obtain a (uniquely determined) natural transformation $\sigma: {\sf Id}_{\BB}\lra EI$ with $I\sigma=(\pi I)^{-1}$, which implies $\pi I\circ I\sigma=1_I$; furthermore, $\sigma$ is an isomorphism since $I\sigma$ is one. 

Conversely, having $E, \pi, \sigma$ as in the definition of full left semi-adjointness, the equation $\pi I\circ I\sigma=1_I$ shows that, with $\sigma$, also $\pi I$ must be an isomorphism. That $I$ is necessarily fully faithful in this situation may be shown as for left adjoint functors with isomorphic adjunction unit: given $f:IA\to IB$ in $\CC$ with $A,B\in|\BB|$, $g\df\sigma_B^{-1}\circ Ef\circ \sigma_A$ is the only $\BB$-morphism $A\to B$ with $Ig=f$.

When $I$ is left adjoint, one has the second triangular identity $E\pi\circ\sigma E=1_E$ and, hence, $1_{IE}=IE\pi\circ I\sigma E=IE\pi\circ(\pi IE)^{-1}$. Conversely, from $IE\pi=\pi IE$ one obtains $I(E\pi\circ \sigma E)=\pi IE\circ I\sigma E=1_IE=I(1_E)$ and, hence, $E\pi\circ \sigma E=1_E$.
\sqs

\noindent Note that, for a fully left semi-adjoint functor $I$, the functor $E$ and the transformation $\pi$  used to define the term will generally {\em not} be unique (up to isomorphism): having $\tilde{E}, \tilde{\pi}$ with corresponding properties gives us natural transformations $\gamma:E\lra \tilde{E}$ and $\delta:\tilde{E}\lra E$ with $\tilde{\pi}\circ I\gamma=\pi,\;\pi\circ I\delta=\tilde{\pi}$, but they will generally be inverse to each other only when $I$ is left adjoint.

We may now rephrase Proposition \ref{functorial}, as follows:

\medskip

\noindent{\bf Corollary.} {\em For a full subcategory $\BB$ of $\CC$, there is a $(\BB,\CC)$-covering class if, and only if, the inclusion functor is fully  left semi-adjoint.}
\end{nist}

\section{The de Vries and Fedorchuk  dualities revisited}
In this section we recall and extend various facts leading up to the Fedorchuk Duality Theorem \cite{F} and sketch how originally it has been derived from the de Vries Duality Theorem \cite{deV}. Our alternative proof of the Fedorchuk Duality Theorem, which avoids the use of the de Vries Duality Theorem, follows in the next section.

\begin{nist}\label{skelnewcor}
\rm
We begin by recalling some statements about skeletal and quasi-open mappings.
It is well known that {\em a mapping  $f:X\lra Y$  of topological spaces is skeletal if,
and only if,
$$\int(\cl (f(U)))\nes$$
for every non-empty open subset  $U$  of} $X$ (see, for example, \cite{D2009}).
Hence, {\em every quasi-open mapping is skeletal.}
Also, {\em  if $f:X\lra Y$ is continuous, then
 $f$ is a skeletal mapping if, and only if,
for every} $F\in \RC(X)$, $\cl(f(F))\in \RC(Y)$ (see, e.g., \cite{D2009}).

\smallskip

We also recall the following result of Blaszczyk \cite{Bl1}:

\medskip

\noindent{\bf Lemma.}
{\em A  continuous mapping $f:X\lra Y$ of topological spaces, is skeletal if, and only if,
for every open dense subset $V$ of\/ $Y$, $\cl_X(f\inv(V))= X$.}

\medskip

Note that {\em every closed irreducible mapping $f:X\lra Y$ is quasi-open}; indeed, by a result of Ponomarev \cite{Pon1},
{\em for such mappings one has that, for every non-empty open subset $U$ of $X$,
\begin{equation}\label{Ponomarev}
f^{\rm \#}(U)\ \df\{y\in Y\st f\inv(y)\sbe U\}
\end{equation}
is a non-empty open subset of} $Y$.
\end{nist}

\begin{lm}\label{qopl}
Let $f:X\lra Y$ and $g:Y\lra Z$ be continuous maps of topological spaces, with  $g\circ f$ and $f$ quasi-open and $\cl(f(X))=Y$.
Then  $g$ is also quasi-open.
\end{lm}

\doc Let $V$ be a non-empty open subset of $Y$. Then $V\cap f(X)\nes$ and thus $U\df f\inv(V)$ is a non-empty open subset of $X$. Hence $W\df \int((g\circ f)(U))$ is a non-empty open subset of $Z$.
Since $f(U)=f(f\inv(V))\sbe V$, we obtain that $g(f(U))\sbe g(V)$. Hence $W\sbe \int(g(V))$. Therefore, $\int(g(V))\nes$. This shows that the mapping $g$ is quasi-open.
\sqs

\begin{nist}\label{begen}
\rm

 It is well known that, {\em for compact Hausdorff spaces $X$ and $Y$, if $EX$ and $EY$ denote their absolutes
 and $\pi_X:EX\lra X$, $\pi_Y:EY\lra Y$ their projective covers, respectively, then, for every continuous mapping
 $f:X\lra Y$, there exists a continuous mapping $\baf:EX\lra EY$ such that} $f\circ \pi_X=\pi_Y\circ\baf$; furthermore,
  {\em the mapping $f$ is surjective if, and only if, the mapping $\baf$ is surjective} (see, e.g., \cite{Gle}, \cite[10M]{Wa}
   and \cite{HJ}). Indeed, the first assertion follows from the fact
 that $\pi_Y$ is a perfect surjective mapping and $EX$ is a projective object of the category $\HC$
 (see the Gleason Theorem \ref{nipr}); the second one is an easy consequence of the irreducibility of the
 mapping $\pi_Y$. Further, Bereznitskij (as cited in \cite{PS}) proved that {\em if $f$ is a continuous surjection,
 then,  the mapping $f$ is quasi-open \tiff the mapping $\baf$ is open.}
 We now prove that this result is true even without the assumption that $f$ be surjective.

\medskip

\noindent {\bf Proposition.} {\em The mapping $f$ is quasi-open \tiff the mapping $\baf$ is open.}

\medskip

\doc ($\Rightarrow$) Since $f$ and $\pi_X$ are quasi-open mappings, the composite map $f\circ\pi_X$ is also
quasi-open; this means: $\pi_Y\circ\baf$ is
quasi-open. Using Blaszczyk's Lemma \ref{skelnewcor},
we now show  that the mapping $\baf$ is skeletal.
So, let $U$ be an open dense subset of $EY$. We have to show that
$\baf\inv(U)$ is a dense subset of $EX$.
By (\ref{Ponomarev}), $V\df (\pi_Y)^\#(U)$,
is a non-empty open subset of $Y$,
and $(\pi_Y)\inv(V)\sbe U$. Using a result of Ponomarev
\cite{Pon1} (see  \cite[Proposition 2, page 345]{Alex}), we
obtain $\pi_Y(\cl(U))=\cl(V)$, {\em i.e.}, $V$ is an open dense
subset of $Y$. Since the mapping $\pi_Y\circ\baf$ is quasi-open,
Blaszczyk's Lemma \ref{skelnewcor} implies that the set
$(\pi_Y\circ\baf)\inv(V)$ is dense in $EX$. With
$(\pi_Y)\inv(V)\sbe U$ we now obtain
$EX=\cl(\baf\inv((\pi_Y)\inv(V)))\sbe \cl(\baf\inv(U))\sbe EX$.
Thus, $\baf\inv(U)$ is a dense subset of $EX$. So, $\baf$ is a
skeletal mapping, and we obtain that the mapping
$\baf$, being closed, is quasi-open (see, e.g., \cite{D2009}). Consequently,  if $F\in \RC(EX)$, then $\baf(F)\in\RC(EY)$. The
spaces $EX$ and $EY$ are extremally disconnected and, thus,
$\RC(EX)=\CO(EX)$ and $\RC(EY)=\CO(EY)$. Since $\CO(EX)$ is a base
for $EX$, we obtain that $\baf$ is an open mapping.

\medskip

\noindent ($\Leftarrow$) Since $\baf$ and $\pi_Y$ are quasi-open mappings, the composite map $\pi_Y\circ\baf=f\circ\pi_X$ is also quasi-open. Since $\pi_X$ is surjective, we conclude  with Lemma \ref{qopl} that $f$ is quasi-open.
\sqs
\end{nist}

We also recall the following theorem of Henriksen and Jerison \cite{HJ}:

\begin{theorem}\label{HJTH}
Let $X$ and $Y$ be compact Hausdorff spaces, $\pi_X:EX\lra X$ and $\pi_Y:EY\lra Y$ be their projective covers,
and $f:X\lra Y$ be a continuous surjection. There exists a unique continuous mapping $\baf:EX\lra EY$ satisfying
$f\circ \pi_X=\pi_Y\circ \baf$ \tiff
%
\begin{equation}\label{skehj}
\ \cl(\int(f\inv(F)))=\cl(f\inv(\int(F))) \mbox{ for every } F\in
\RC(Y).
\end{equation}
\end{theorem}

\begin{rem}\label{remhj}
\rm
We note that, in Lemmas 1 and 3 of \cite{HJ}, the expression $``(\pi_X)\inv(\a)"$ should be replaced by $``\cl((\pi_X)\inv(\int(\a)))$''.
Indeed, supposing that $(\pi_X)\inv(\a)$ is open for every $\a\in \RC(X)$, we obtain, by the result of Ponomarev \cite{Pon1} cited above (see (\ref{Ponomarev})), that $(\pi_X)^{\rm \#}((\pi_X)\inv(\a))$ is open, {\em i.e.}, that
$\a$ is open for every $\a\in \RC(X)$, which is true only when $X$ is extremally disconnected. Fortunately, all other statements of \cite{HJ} remain true, although their proofs have to be slightly adjusted.

\medskip

Clearly, {\em every continuous skeletal mapping $f:X\lra Y$ of topological spaces $X$ and $Y$ satisfies}
(\ref{skehj}) (\cite{MR}). Consequently,
{\em every
quasi-open mapping $f:X\lra Y$ satisfies} (\ref{skehj}) (\cite{HJ, PS}).
\end{rem}

Of great importance to our investigations  is the following beautiful theorem by Alexandroff \cite{Alex}, which follows easily from Ponomarev's results \cite{Pon1} on irreducible mappings:

  \begin{theorem}\label{thalpon}{\rm (\cite[Corollary, p. 346]{Alex})}
  Let $p:X\lra Y$ be a closed irreducible mapping. Then the map
  $$\p_p:\RC(X)\lra \RC(Y),\ \ H\mapsto p(H).$$
  is a Boolean isomorphism, and one has $\p_p\inv(K)=\cl_X(p\inv(\int_Y(K)))$, for all $K\in\RC(Y)$.
  \end{theorem}

We also need the next  assertion, which is similar to that of Theorem \ref{HJTH}; it follows from a much more general theorem of  \v{S}apiro \cite{Sa} (see also Uljanov \cite{Ul}). For the sake of completeness of our exposition, we outline a proof.

\begin{pro}\label{HJTHM}
Let $X$ and $Y$ be compact Hausdorff spaces
and $f:X\lra Y$ a quasi-open mapping. Then, with $\pi_X:EX\lra X$ and $\pi_Y:EY\lra Y$ the projective covers of $X$ and $Y$, there exists a unique continuous mapping $\baf:EX\lra EY$ such that
$f\circ \pi_X=\pi_Y\circ \baf$.
\end{pro}

\doc  The existence of such a mapping $\baf:EX\lra EY$ was already established in \ref{begen}.
 Suppose that $g:EX\lra EY$ is any continuous mapping with $f\circ \pi_X=\pi_Y\circ g$.
 Then, by Proposition \ref{begen}, $g$ is open, whence $g(EX)\in\CO(EY)$.
 Since $f$ is quasi-open, we have $Z\df f(X)\in\RC(Y)$.   Clearly, $W\df\cl_{EY}((\pi_Y)\inv(\int_Y(Z))) \in\CO(EY)$ and, hence, $W$ is an extremally disconnected compact Hausdorff space.
 Since $\pi_Y$ is a closed surjection, we obtain that $\pi_Y(W)=Z$ and put $\pi_Z\df\pi_Y\!\!\upharpoonright_W:W\lra Z$.
 We claim that $\pi_Z$ is irreducible. Indeed, suppose that there exists a closed proper subset $F$ of $W$ such
 that $\pi_Z(F)=Z$. Then $F$ and $F\cup (EY\stm W)$ are proper closed subsets of $EY$, and $\pi_Y(F\cup (EY\stm W))=Y$.
 Since $\pi_Y$ is irreducible, we obtain a contradiction. Hence, $W$ is the absolute of $Z$, and $\pi_Z$ is the
 projective cover of $Z$. Let $\a:\RC(EY)\lra \RC(Y)$ be defined by the formula $\a(G)\df\pi_Y(G)$,
 for every $G\in\RC(EY)(=\CO(EY))$. Then, by Theorem \ref{thalpon}, $\a$ is a Boolean isomorphism.
Clearly, $\a(W)=Z$. Since
$\a(g(EX))=\pi_Y(g(EX))=f(\pi_X(EX))=f(X)=Z$, we obtain
 $g(EX)=W$. Now, applying Theorem \ref{HJTH} for $\pi_X$, $\pi_Z$ and $f\upharpoonright_X:X\lra Z$ (and noting that
 $f\upharpoonright_X$ is quasi-open),
 we conclude $g=\baf$.
 \sqs

\begin{nist}\label{qzhcth}
\rm
The following theorem was proved in \cite[Corollary 3.2(c)]{D-PMD12} and \cite[Corollary 2.4(c)]{D-a0903-2593}:

\medskip

\noindent{\bf Theorem.}{\rm (\cite{D-PMD12})}
{\em The restrictions of the functors $S^t$ and $S^{a}$ of the Stone Duality Theorem {\em(see \ref{Std})} render the category\/ $\Stone_{\rm{qop}}$ of
Stone spaces  and quasi-open mappings as dually equivalent to the
category $\Bool_{\rm{sup}}$ of  Boolean algebras and suprema-preserving
Boolean homomorphisms.}

\medskip

Using the well known fact that {\em complete Boolean algebras correspond  to extremally disconnected compact Hausdorff spaces under the Stone duality,} and arguing as at the end of the first part of the proof of Proposition \ref{begen} ({\em i.e.}, using the fact that {\em a continuous mapping between two extremally disconnected compact Hausdorff spaces is quasi-open if, and only if, it is open}\/), under a further restriction of the Stone duality, we obtain the following corollary:

\medskip

\noindent{\bf Corollary.}
{\em The category\/ $\bf{ECH}_{\rm{op}}$ of extremally disconnected  compact
Hausdorff spa\-ces  and open mappings is dually equivalent to the
category\/ $\CBool_{\rm{sup}}$ of  complete Boolean algebras and suprema-preserving
Boolean homomorphisms.}
\end{nist}

Before we formulate and prove the Fedorchuk Duality Theorem, it is useful for us to recall the de Vries Duality Theorem.

\begin{defi}\label{dval}{\rm (De Vries \cite{deV})}
\rm
We denote by $\DHC$ the category of complete normal contact algebras (see \ref{conalg}); its morphisms $\p:(A,C)\lra (A\ap,C\ap)$ are maps $A\lra A\ap$
satisfying the conditions:

\smallskip

\noindent(DV1) $\p(0)=0$;\\
(DV2) $\p(a\we b)=\p(a)\we \p(b)$, for all $a,b\in A$;\\
(DV3) If $a, b\in A$ and $a\ll_C b$, then $(\p(a^*))^*\ll_{C\ap}
\p(b)$;\\
(DV4) $\p(a)=\bigvee\{\p(b)\st b\ll_{C} a\}$, for every $a\in
A$;

\medskip

{\noindent}the composition $``\diamond$" of
$\p_1:(A_1,C_1)\lra (A_2,C_2)$ with $\p_2:(A_2,C_2)\lra (A_3,C_3)$
in $\DHC$ is defined by
\begin{equation}\label{diamc}
\ \p_2\diamond\p_1 \df (\p_2\circ\p_1)\cuk,
\end{equation}
 where, for objects $(A,C),(A\ap,C\ap)$ in $\DHC$ and any
function $\psi:A\lra A\ap$, one defines
$\psi\cuk:(A,C)\lra (A\ap,C\ap)$ for all $a\in A$ by
\begin{equation}\label{cukfc}
\ \psi\cuk(a)\df \bigvee\{\psi(b)\st b\ll_{C} a\}.
\end{equation}

We call the morphisms of the category $\DHC$ {\em de Vries morphisms}.
\end{defi}

\begin{fact}\label{dvmrem}{\rm (\cite{deV})}
Let $\p:(A,C)\lra(A\ap,C\ap)$ be a  de Vries morphism. Then:

\smallskip

\noindent{\em (a)} $\p(1_A)=1_{A\ap}$;

\smallskip

\noindent{\em (b)} for every $a\in A$, $\p(a^*)\le (\p(a))^*$;

\smallskip

\noindent{\em (c)} for every  $a,b\in A$, $a\ll_C b$ implies
$\p(a)\ll_{C\ap}\p(b)$;

\smallskip

\noindent{\em(d)} if $\p\ap:(A\ap,C\ap)\lra(A'',C'')$ is a  de Vries morphism, such that $\p\ap$ is a suprema-preserving Boolean homomorphism, then $\p\ap\di\p=\p\ap\circ\p$.
\end{fact}

De Vries \cite{deV} proved the following duality theorem:

\begin{theorem}\label{dvth}{\rm (\cite{deV})}
The categories $\HC$ and $\DHC$ are dually equivalent.
\end{theorem}

 \noindent{\em Sketch of the proof.}~ One defines contravariant functors $$\Psi^t:\HC\lra\DHC ,\quad \Psi^a:\DHC\lra\HC, $$  by
 \begin{itemize}
  \item $\Psi^t(X,\tau)\df (\RC(X,\tau),\rho_X),$ for all
$X\in\card\HC$;

\item $\Psi^t(f)(G)\df \cl(f\inv(\int(G))),$ for all
$f\in\HC(X,Y)$ and $G\in \RC(Y)$;
  \item $\Psi^a(A,C)\df (\CL(A,C),\TT),$ for all $(A,C)\in\card\DHC$,
where $\TT$ is the topology on $\CL(A,C)$ having
the family $\{\ups_{(A,C)}(a)\st a\in A\}$ with $\ups_{(A,C)}(a) =
\{\s \in \CL(A,C)\st  a \in \s\}$ as a base of closed sets;

\item $\Psi^a(\p)(\s\ap)\df \{a\in A\st \forall\, b\in \!A\; (\,b\ll_C a^* \Longrightarrow (\p(b))^*\in\s\ap\,)\},$ for all
$\p\in\DHC((A,C),(A\ap,C\ap))$ and
$\s\ap\in\CL(A\ap,C\ap)$.
\end{itemize}

\noindent Then one shows that, for every $(A,C)\in\card\DHC$,
$\up_{(A,C)}:(A,C)\lra\Psi^t(\Psi^a(A,C))$
is a $\DHC$-isomorphism, producing the natural isomorphism
$$\ups:{\sf Id}_{\,\DHC}\lra\Psi^t\circ\Psi^a.$$
Likewise,
$$t\ap:{\sf Id}_{\,\HC}\lra\Psi^a\circ\Psi^t,$$ with
$t\ap_{X}(x)\df \s_x$,
  for every $X\in\card\HC$ and all $ x\in X$,
is a natural isomorphism. Thus, the categories $\HC$ and $\DHC$  are dually equivalent.
\sqs

\medskip

\noindent We note that, in \cite{deV}, de Vries used regular open
sets, rather than regular closed sets, as we do here. Hence, above we have paraphrased
his definitions in terms of regular
closed sets.

\begin{rem}\label{dvasim}
  If $\p\in\DHC((A,C),(A\ap,C\ap))$ and $\p$ is a Boolean homomorphism, the definition of the contravariant functor $\Psi^a$ on
 morphisms may be simplified {\em (see \cite{D2009})},  as follows: in the notation of {\em \ref{uniqult}}, one puts
 $$\Psi^a(\p)(\s_{u\ap})\df\s_{\p\inv(u\ap)},$$
 for every
 ultrafilter  $u\ap$ in $A\ap$.

{\rm Indeed, since $\p$ is a Boolean homomorphism, we obtain that $\p\inv(u\ap)$ is an ultrafilter in $A$. Let now $a\in\Psi^a(\p)(\s_{u\ap})$, and suppose that
$a\nin\s_{\p\inv(u\ap)}$. Then there exists $b\in\p\inv(u\ap)$
such that $a(-C)b$. Thus, $b\ll a^*$. Then there is $c\in A$
with $b\ll c\ll a^*$. Since $a\in\Psi^a(\p)(\s_{u\ap})$, we obtain
$(\p(c))^*\in\s_{u\ap}$. On the other hand, we have that
$\p(b)\in u\ap$ and, by Fact \ref{dvmrem}(c), $\p(b)\ll\p(c)$,
i.e., $\p(b)(-C\ap)(\p(c))^*$, a contradiction. Hence,
$a\in\s_{\p\inv(u\ap)}$. This shows that
$\Psi^a(\p)(\s_{u\ap})\sbe \s_{\p\inv(u\ap)}$. Therefore, by Fact \ref{fact29},
$\Psi^a(\p)(\s_{u\ap})= \s_{\p\inv(u\ap)}$.}
\end{rem}

\begin{nist}\label{dcompn}
\rm We now formulate the Fedorchuk Duality Theorem. Let
$\HC_{\rm{qop}}$ be the category of  compact Hausdorff spaces and  their quasi-open mappings, and
$\Fed $ the category of complete normal
contact algebras whose morphisms $\p:(A,C)\lra (A\ap,C\ap)$ 
are all suprema-preserving Boolean homomorphisms $\p:A\lra A\ap$
satisfying the following condition

\medskip

\noindent (F) $\p(a)C\ap\p(b)$ implies
$aCb$, for all $a,b\in A$.
\bigskip

Since $\p$ preserves the negation, we see
immediately that condition (F) is equivalent to asking that

\medskip

\noindent(F$'$) $a\ll_C b$ implies
$\p(a)\ll_{C\ap}\p(b)$, for all $a,b\in A$.

\medskip

The Fedorchuk Duality Theorem states:

\medskip

\noindent{\bf Theorem.} {\rm (Fedorchuk \cite{F})}
{\em The categories $\HC_{\rm{qop}}$ and $\Fed $ are dually equi\-valent.}

\bigskip

In \cite{F}, Fedorchuk proves that the category $\Fed$ is a (non-full)
subcategory of the category $\DHC$, $\Psi^t(\HC_{\rm{qop}})\sbe \Fed$ and $\Psi^a(\Fed)\sbe
\HC_{\rm{qop}}$, where $\Psi^t$ and $\Psi^a$ are de Vries'
contravariant functors (see Theorem \ref{dvth}). Then, applying the de
Vries Duality Theorem, he concludes that the categories
$\HC_{\rm{qop}}$ and $\Fed $ are dually equivalent.

For later use we fix the notation $\Phi^t\df \Psi^t|_{\HC_{\rm{qop}}}$ and $\Phi^a\df \Psi^a|_{\Fed}$ for the restrictions of de Vries' functors, so that we have the contravariant functors

\begin{center}
$\Phi^t:\HC_{\rm{qop}}\lra\Fed \mbox{ and } \Phi^a:\Fed\lra\HC_{\rm{qop}}.$
\end{center}

\medskip
\end{nist}

\section{A new approach to the Fedorchuk Duality}

Applying Proposition \ref{mere duality}, we now embark on providing
an alternative proof of the Fedorchuk Duality Theorem, {\em without} making use of the
de Vries Duality Theorem. In this way, we
will also obtain a topological interpretation of all algebraic notions
used in the Fedorchuk Duality Theorem.

\begin{nist}\label{n1n}
\rm
 In view of Section 3, {\em throughout this section we use the following notation:}
$$\AA\df\CBool_{\rm{sup}}, \ \ \BB\df\bf{ECH}_{\rm{op}}, \ \
\CC\df\HC_{\rm{qop}},$$
{\em  with $I:\BB\hookrightarrow\CC$ denoting the inclusion functor;
 $\PP$ denotess the class of all irreducible continuous maps between compact Hausdorff spaces with domain in $\BB$.}

Trivially, $\BB$ is a full
subcategory of $\CC$ that is closed under $\CC$-isomorphisms.
Less trivially, by \cite{Pon1}, we know that the class $\PP$ is actually a class of $\CC$-morphisms.
The class $\PP$ obviously satisfies conditions (P1-3) of Section 3.
 From \ref{niab}(c) we know that condition (P4)
is also satisfied. Finally,
Propositions \ref{HJTHM} and \ref{begen} show that
condition (P5$^*$)  is fulfilled and, hence, also
 condition (P5) (and (P4$^*$)).
  Thus, Corollary \ref{stronger conditions} confirms the following fact:

\medskip

\noindent{\bf Proposition.}\ \ {\em  The category\ \ $\bf{ECH}_{\rm{op}}$\ \ is a full coreflective subcategory of the category\ \ $\HC_{\rm{qop}}$.}

\medskip

With the restrictions $$T\df S^a\upharpoonright_{\AA}\ \ \mbox{ and }\ \ S \df S^t\upharpoonright_{\BB}$$ of the functors furnishing the Stone Duality, using Corollary \ref{qzhcth} we obtain the contravariant functors $T:\AA\lra\BB$ and $S:\BB\lra\AA$. Together with the restrictions $\eta\df t\upharpoonright_{\BB}$ and $\ep\df s\upharpoonright_{\AA}$ of Stone's natural isomorphisms (so that one has natural isomorphisms $\eta: {\sf Id}_{\BB}\lra T\circ S$
 and $\ep: {\sf Id}_{\AA}\lra S\circ T$), they realize a dual equivalence between the categories $\AA$ and $\BB$.

 Defining the category $\DD$ as in Proposition \ref{mere duality}, we obtain the
 full embedding $J:\AA\lra\DD$  and the dual equivalence $\tilde{T}:\DD\lra\CC$ which extends the dual
  equivalence   $T:\AA\lra \BB$, so that $I\circ T=\tilde{T}\circ J$, as given by Proposition \ref{mere duality}. We now prove that the categories $\Fed $ and $\DD$
  are equivalent, thus completing our alternative proof of the Fedorchuk Duality Theorem.
 In doing so, using  very natural topological arguments, we explain how the algebraic properties of the objects and morphisms of the category $\Fed$ correspond via the functors $\Phi^a$ and  $\Phi^t$ of \ref{dcompn}
 to their topological counterparts.
 \end{nist}

 We start by proving a simple lemma.

 \begin{lm}\label{lem1}
 Let $D=(A,\pi)\in|\DD|$, and for every $a,b\in A$, define
 $$aC_D b\Leftrightarrow \pi(s_A(a))\cap\pi(s_A(b))\nes.$$
 Then $(A,C_D)$ is a complete normal contact algebra.
 \end{lm}

 \begin{proof}
 We have that  $A\in|\AA|$, $\pi\in\PP$, $\pi: T(A)\lra X$   and $X\in|\CC|$.
Clearly,  $T(A)=S^a(A)$, $\CO(T(A))=\RC(T(A))$, and $\pi$ is a closed irreducible map. Thus, by Theorem \ref{thalpon},
$\p_\pi\circ s_A:A\lra \RC(X)$ is a Boolean isomorphism. Here $(\p_\pi\circ s_A)(a)=\pi(s_A(a))$ and, hence,
$a C_D b \Leftrightarrow (\p_\pi\circ s_A)(a)\cap (\p_\pi\circ s_A)(b)\nes,$ for all $a,b\in A$. Since $(\RC(X),\rho_X)$ is a CNCA,
we obtain that $(A,C_D)$ is also a CNCA.
(Briefly: we have just transported the NCA-structure on $\RC(X)$ to $A$ using Boolean isomorphism $\p_\pi\circ s_A$.)
\end{proof}

Clearly, if $D=(A,\pi)\in|\DD|$, then the definition of the relation $C_D$ from Lemma \ref{lem1} can be expressed in the following equivalent form:
for all $a,b\in A$,
\begin{equation}\label{CD}
aC_D b\iff \ex\, u,v\in\ult(A):\, a\in u, b\in v\mbox{ and }\pi(u)=\pi(v).
\end{equation}

 \begin{lm}\label{lem2}
 Let $(A,C)$ be a CNCA and $R_{(A,C)}$ be the equivalence relation of {\em Definition \ref{relr}} (see also {\em Proposition \ref{relrult}(b)}), i.e.,
  for all $u,v\in T(A)$, $$uR_{(A,C)}v\Leftrightarrow u\times v\sbe C.$$
  Then the natural quotient mapping  $\pi_{(A,C)}:T(A)\lra T(A)/R_{(A,C)}$ is an irreducible mapping, and $T(A)/R_{(A,C)}$ is a compact Hausdorff space, i.e., $\pi_{(A,C)}\in\PP$.
 \end{lm}

  \begin{proof}
  For brevity,  we put $Y\df T(A)$,
   $s\df s_A$,  $R\df R_{(A,C)}$, $\pi\df \pi_{(A,C)}$
   and $X\df Y/R$.
First we prove that the mapping $\pi$ is irreducible. Suppose
that there exists a proper closed subset $G$ of $Y$ such that
$\pi(G)=X$, and consider $u\in Y\stm G$. Since $G$ is compact subset of $Y$
and $s:A\lra \CO(Y)$ is a Boolean isomorphism, there exists $a\in
A$ such that $G\sbe s(a)\sbe Y\stm\{u\}$. Hence $a\neq 1$; thus
$a^*\neq 0$. Then, by condition \ref{di6} of \ref{def:ll}, there exists $b\neq 0$ such
that $b\ll a^*$, {\em i.e.}, $b(-C)a$. By \cite[Corollary 2.17]{kop89},
there exists an ultrafilter $v$ in $A$ such that $b\in v$. Then
$a^*\in v$ and, hence, $a\nin v$. Then $v\nin s(a)$ and, therefore,
$v\in Y\stm G$ follows. There now exists $w\in G$ such that
$\pi(w)=\pi(v)$, {\em i.e.}, $vRw$. Since $w\in G$, we obtain $w\in
s(a)$, {\em i.e.}, $a\in w$. So, we have $a\in w$, $b\in v$, with
$v\times w\sbe C$, which implies $aCb$: a contradiction. This shows that
the mapping $\pi$ is irreducible.

 Next we show that, for every $u\in Y$, the equivalence class $[u]$ of $u$ is a closed subset of $Y$.
 Indeed, let $u\in Y$ and $v\in Y\stm [u]$. Then $u(-R)v$, {\em i.e}, $u\times v\nsubseteq C$. Consequently, there exist $a\in u$
 and $b\in v$ such that $a(-C)b$. Then $v\in s(b)$ follows. Also, $s(b)\cap [u]=\ems$. Indeed, if $w\in s(b)$ then $b\in w$,
 and since $a(-C)b$, we obtain $u(-R)w$, i.e., $w\nin [u]$. Therefore, $[u]$  is a closed subset of $Y$.

We now prove that $R$ is a closed equivalence relation on $Y$ (in the sense of \cite[2.4.9]{E}). To this end,
we let $U$ be  an open  subset of $Y$ and must prove that the union of all equivalence classes that are contained
in $U$ is open in $Y$. Let $u\in Y$ be such that $[u]\sbe U$. Since $[u]$ is compact, there exists $a\in A$ such
that $[u]\sbe s(a)\sbe U$. Then $a\in u$.
For every $M\sbe Y$, we set $$[M]\df \bigcup\{[v]\st v\in M\}.$$
With $V\df Y\stm [s(a^*)]$ we have $[u]\sbe V$. Indeed, we certainly have
$[u]\cap s(a^*)=\ems$. Suppose that $[s(a^*)]\cap [u]\nes$. Then
there exists $v\in s(a^*)$ such that $[v]\cap [u]\nes$. Thus
$[v]=[u]$. Since $[u]\cap s(a^*)=\ems$, we obtain that $v\in
[v]\cap s(a^*)=\ems$, a contradiction. So, $[u]\sbe V$.

Next, we prove that $V$ is open and consider $v\in V$. Suppose that, for every $b\in v$, $s(b)\nsubseteq V$;
then, $s(b)\cap [s(a^*)]\nes$. Hence, for every $b\in v$ there exist $v_b\in s(b)$ and
$w_b\in s(a^*)$ such that $v_b\in [w_b]$. This means that,
for every $b\in v$ there exist $v_b,w_b\in Y$ such that $b\in v_b$, $a^*\in w_b$ and $v_b R w_b$. Using
Proposition \ref{relrult}(a), we obtain $a^* C b$ for all $b\in v$. Thus, by Corollary \ref{uniqult},
$a^*\in\s_v$, where $\s_v$ is the cluster generated by $v$.
Now, Theorem \ref{conclustth} gives us that there exists an ultrafilter $w$ in $A$ such that $a^*\in w$ and $\s_w=\s_v$.
Hence $v\cup w\sbe \s_v$ and thus, $v\times w\sbe C$, {\em i.e.}, $vRw$. Therefore, $[v]=[w]$. Since $w\in s(a^*)$,
we obtain that $v\in [s(a^*)]=Y\stm V$, a contradiction. Consequently, there exists $b\in v$ such that $s(b)\sbe V$.
Since $v\in s(b)$, we obtain that $V$ is an open subset of $Y$.

Finally, we establish that $V$ is a subset of the union of all equivalence classes that are contained
in $U$. Let $w\in Y$ and $[w]\cap V\nes$. Suppose that $[w]\nsubseteq U$. Then $[w]\nsubseteq s(a)$. Hence,
there exists $v\in (Y\stm s(a))\cap [w]$. Then $v\in s(a^*)$ and $[w]=[v]\sbe [s(a^*)]= Y\stm V$, a contradiction.
Hence, $[w]\sbe U$.

So, $R$ is a closed relation. Then, by the Alexandroff Theorem \cite[Theorem 3.2.11]{E}, $X$ is a compact
 Hausdorff space, {\em i.e.}, $X\in|\CC|$. Since $\pi$ is a closed irreducible mapping, we obtain that $\pi$ is
 quasi-open (see  \ref{skelnewcor}). Hence, $\pi\in \CC(Y,X)$ follows. Since also $Y=T(A)\in|\BB|$,
 we conclude $\pi\in\PP$.
  \end{proof}

  Following \cite{BBSV}, we call a closed equivalence relation $R$ on a compact Hausdorff
space $X$ {\em irreducible}\/ if the natural quotient mapping $\pi_R: X \lra X/R$ is irreducible.

 \begin{pro}\label{mainpro}
 For a complete Boolean algebra $A$, let ${\rm NCRel}(A)$ be the set of all normal contact relations on $A$  and ${\rm IRel}(T(A))$ the set of all closed irreducible equivalence relations on $T(A)$.
 Then the function  $$f:{\rm NCRel}(A)\lra {\rm IRel}(T(A)), \ \ C\mapsto R_{(A,C)},$$ is a bijection, and $f\inv(R)= C_{(A,\pi_R)}$, for every $R\in{\rm IRel}(T(A))$.
 \end{pro}

  \begin{proof}
    Lemma \ref{lem2} shows that $R_{(A,C)}$ is an irreducible relation on $T(A)$. So, the function $f$ is well-defined. The map $$g:{\rm IRel}(T(A))\lra {\rm NCRel}(A),\ \ R\mapsto C_{(A,\pi_R)},$$ is also well-defined by Lemma \ref{lem1} since, obviously, $(A,\pi_R)\in|\DD|$.

 Consider any $C\in {\rm NCRel}(A)$ and set $R\df f(C)$, $\pi\df \pi_R$ and $D\df (A,\pi)$. Then $g(R)=C_D$. Hence, by (\ref{CD}), we have
  $$aC_D b\Leftrightarrow \ex\, u,v\in\ult(A)\; ( a\in u,\, b\in v\, {\rm and}\, \pi(u)=\pi(v)\,).$$
   Clearly, $\pi(u)=\pi(v)\Leftrightarrow uRv$. Thus, using Proposition \ref{relrult}(a), we obtain $$aC_D b \Leftrightarrow \ex u,v\in\ult(A)\; (a\in u,\, b\in v\;{\rm and}\;uR_{(A,C)}v) \Leftrightarrow aCb\,.$$ This shows $g(f(C))=C$.

 Now consider any $R\in {\rm IRel}(T(A))$. For $C\df g(R)$ one then has $f(C)=R_{(A,C)}$. In order to show $R=R_{(A,C)}$, we let $u,v\in T(A)$ and assume $uRv$. Then $\pi_R(u)=\pi_R(v)$, and we set $x\df \pi_R(u)$. Then $x\in\pi_R(s(a))\cap\pi_R(s(b))$, for all $a\in u$ and $b\in v$. This shows $u\times v\sbe C_{(A,\pi_R)}$, and $uR_{(A,C)}v$ follows. Conversely, consider $u,v\in T(A)$ with $uR_{(A,C)}v$. Assuming $u(-R)v$ we obtain $\pi_R(u)\neq\pi_R(v)$. Since $R$ is a closed equivalence relation, the space $T(A)/R$ is Hausdorff. Thus, there exist open disjoint neighborhoods $U$ and $V$ of $\pi_R(u)$ and $\pi_R(v)$, respectively. Then $u\in\pi_R\inv(U)$ and $v\in\pi_R\inv(V)$. Consequently, there exist $a\in u$ and $b\in v$ such that $s(a)\sbe \pi_R\inv(U)$ and $s(b)\sbe \pi_R\inv(V)$. This obviously implies that $\pi_R(s(a))\cap\pi_R(s(b))=\ems$. We obtain $a(-C_{(A,\pi_R)})b$, which is a contradiction because $u\times v\sbe C_{(A,\pi_R)}$. Hence, $uRv$ follows. This completes the proof of
  $f(g(R))=R$.
   \end{proof}

We set $|\DD|_{\rm nqm}\df \{(A,\pi)\in|\DD|\st \pi \mbox{ is a natural quotient mapping}\}$ and obtain, in the notation of Lemma \ref{lem2}:

\begin{cor}\label{cor1f}
The correspondence $F:|\Fed|\lra |\DD|_{\rm nqm}, (A,C)\mapsto (A,\pi_{(A,C)})$, is a bijection.
 \end{cor}

\begin{proof}
Let $(A,C)\in|\Fed|$. Then, by Lemma \ref{lem2}, $\pi_{(A,C)}\in\PP$ and $(A,\pi_{(A,C)})\in|\DD|$. This makes the correspondence $F$ well-defined.
Now, with the notation of Lemma \ref{lem1}, we consider $$G: |\DD|_{\rm nqm}\lra |\Fed|,\ \ (A,\pi)\mapsto (A,C_{(A,\pi)}).$$
Clearly, Lemma \ref{lem1} confirms that $G$ is well-defined. We show that $F$ and $G$ are inverse to each other. 

For $(A,C)\in|\Fed|$ one has $G(F(A,C))=G(A,\pi_{(A,C)})=(A,C_{(A,\pi_{(A,C)})})$. By Proposition \ref{mainpro}, $C=g(f(C))=C_{(A,\pi_{R_{(A,C)}})}$ follows.
Since $\pi_{(A,C)}=\pi_{R_{(A,C)}}$ (see the proof of Lemma \ref{lem2}), we obtain $G(F(A,C))=(A,C)$.

 For $(A,\pi)\in|\DD|_{\rm nqm}$ one has $F(G(A,\pi))=F(A,C_{(A,\pi)})=(A,\pi_{(A,C_{(A,\pi)})})$. Denote by $R_\pi$  the relation on $T(A)$ determined by the fibers of $\pi$; then $R_\pi\in IRel(T(A))$. Using once more Proposition \ref{mainpro}, we obtain $R_\pi=f(g(R_\pi))=R_{(A,C_{(A,\pi_{(R_\pi)})})}$. Since $\pi_{(A,C_{(A,\pi)})}=\pi_{R_{(A,C_{(A,\pi)})}}$ and $\pi=\pi_{(R_\pi)}$ (because $\pi$ is a natural quotient map), we obtain $F(G(A,\pi))=(A,\pi)$.
\end{proof}

Note that Lemmas \ref{lem1} and \ref{lem2}, Proposition \ref{mainpro} and Corollary \ref{cor1f} reveal the topological nature of CNCAs, {\em i.e.}, of the objects of the category $\Fed$. Proposition \ref{mainpro} implies also Bezhanishvili's Theorem \cite[Theorem 8.1]{Bezh} mentioned in the Introduction: for any complete Boolean algebra $B$ there is a bijection between the set of all normal
contact relations on $B$ and the set of all (up to homeomorphism) Hausdorff irreducible images of the Stone dual $S^a(B)$ of $B$. In \cite{Bezh} this result is obtained with the help of de Vries' Duality Theorem, while our proof is direct and therefore topologically more informative.

 \begin{pro}\label{morphpro}
For objects $(A,C),(A\ap,C\ap)$ in $\Fed$, a Boolean homomorphism $\psi:A\lra A\ap$ satisfies condition {\em (F)} of {\em \ref{dcompn}} if, and only if,  $u\ap R_{(A\ap,C\ap)}v\ap$ implies $T(\psi)(u\ap)R_{(A,C)}T(\psi)(v\ap)$, for all  $u\ap,v\ap\in T(A\ap)$.
\end{pro}

\begin{proof}
($\Rightarrow$) \ Let $\psi$ satisfy condition (F) and $u\ap,v\ap\in T(A\ap)$ be such that $u\ap R_{(A\ap,C\ap)}v\ap$.
Suppose that $T(\psi)(u\ap)(-R_{(A,C)})T(\psi)(u\ap)$. Then, for $\pi\df \pi_{(A,C)}$ and  $\pi\ap\df \pi_{(A\ap,C\ap)}$ (see Lemma \ref{lem2} for notation),
we have $\pi\ap(u\ap)=\pi\ap(v\ap)$ and
$\pi(T(\psi)(u\ap))\neq\pi(T(\psi)(v\ap))$. Putting $u\df T(\psi)(u\ap)$ and
$v\df T(\psi)(v\ap)$ we obtain $u=\psi\inv(u\ap)$, $v=\psi\inv(v\ap)$ and
$\pi(u)\neq\pi(v)$.  Since, by Lemma \ref{lem2}, the space $X\df T(A)/R_{(A,C)}$ is Hausdorff, the points
$\pi(u)$ and $\pi(v)$ have disjoint neighborhoods $U$ and $V$,
where $\pi(u)\in U$ and $\pi(v)\in V$. Then there exist $a,b\in A$
such that $u\in s_A(a)\sbe \pi\inv(U)$ and $v\in s_A(b)\sbe
\pi\inv(V)$. Thus $\pi(s_A(a))\cap\pi(s_A(b))=\ems$, {\em i.e.},
$a(-C)b$. Since $\psi$ satisfies condition (F), we obtain
$\psi(a)(-C\ap)\psi(b)$, which means $\pi\ap(s_{A\ap}(\psi(a)))\cap
\pi\ap(s_{A\ap}(\psi(b)))=\ems$. This, however, is impossible,
because $\pi\ap(u\ap)=\pi\ap(v\ap)$, $u\ap\in s_{A\ap}(\psi(a))$
and $v\ap\in s_{A\ap}(\psi(b))$.
Hence, $T(\psi)(u\ap)R_{(A,C)}T(\psi)(u\ap)$.

\medskip

\noindent($\Leftarrow$) Assuming that $\psi$ does not satisfy condition (F), we obtain $a,b\in A$ such that $\psi(a)C\ap\psi(b)$ but $a(-C)b$.
By Proposition \ref{relrult}(a), there exist $u\ap,v\ap\in T(A\ap)$ such that $u\ap R_{(A\ap,C\ap)}v\ap$, $\psi(a)\in u\ap$ and $\psi(b)\in v\ap$.
Then $a\in\psi\inv(u\ap)$ and $b\in\psi\inv(v\ap)$, {\em i.e.}, $a\in T(\psi)(u\ap)$ and $b\in T(\psi)(v\ap)$. By our hypothesis, we have $T(\psi)(u\ap)R_{(A,C)}T(\psi)(u\ap)$. Thus $aCb$, a contradiction. Therefore, $\psi$ satisfies condition (F).
\end{proof}

Note that Proposition \ref{morphpro} reveals the topological nature of the morphisms of the category $\Fed$.

\begin{theorem}\label{thnewfed}
The categories $\Fed$ and $\DD$ are equivalent.
\end{theorem}

\begin{proof}
In Corollary \ref{cor1f} we defined a correspondence
 $F:|\Fed|\lra |\DD|$ by setting $F(A,C)\df (A,\pi_{(A,C)})$ for all $(A,C)\in|\Fed|$.
Now we extend this correspondence to morphisms of $\DD$ to obtain a
  functor $$F:\Fed\lra \DD.$$
Let $\a\in\Fed((A,C),(A\ap,C\ap))$. Then $\a\in\AA(A,A\ap)$, and
$F(A,C)= (A,\pi_{(A,C)})$,\;  $F(A\ap,C\ap)= (A\ap,\pi_{(A\ap,C\ap)})$.
For $\pi\df \pi_{(A,C)}$, $\pi\ap\df \pi_{(A\ap,C\ap)}$, $X\df \cod(\pi)$ and $X\ap\df\cod(\pi\ap)$,  we will define $f\in\CC(X\ap,X)$
such that $f\circ\pi\ap=\pi\circ T(\a)$.

Since $\a$ satisfies condition (F), using Proposition \ref{morphpro}, we obtain that, if  $u\ap,v\ap\in T(A\ap)$ and $\pi\ap(u\ap)=\pi\ap(v\ap)$, then
\begin{equation}\label{comf}
\pi(T(\a)(u\ap))=\pi(T(\a)(v\ap)).
\end{equation}

\noindent To define $f$, since $\pi\ap$ is a surjection, given $x\ap\in X\ap$, one has some $u\ap\in T(A\ap)$ such that
$x\ap=\pi\ap(u\ap)$, and with (\ref{comf}) we can put $$f(x\ap)\df \pi(T(\a)(u\ap)).$$
Since $\pi\ap$ is a
quotient mapping, we obtain that $f:X\ap\lra X$ is continuous. We
claim that $f$
 is quasi-open. Indeed, for $U$ a non-empty open subset of $X\ap$, using the fact that $\pi\ap$ is a
 surjection and $T(\a)$ is open, we obtain that $V\df T(\a)((\pi\ap)\inv(U))$ is a non-empty open
 subset of $T(A)$. Thus, by (\ref{Ponomarev}), $W\df \pi^\# (V)$ is a non-empty open subset of $X$.
 Since $W\sbe \pi(V)=\pi(T(\a)((\pi\ap)\inv(U)))=f(\pi\ap((\pi\ap)\inv(U)))=f(U)$, we obtain that $f$ is a
 quasi-open. Hence, $f\in\CC(X\ap,X)$ and
$f\circ\pi\ap=\pi\circ T(\a)$. Therefore, $(\a,f)\in\DD(F(A,C),F(A\ap,C\ap))$.
Since the function $f$ is uniquely determined by $\a$, we denote it by $f_\a$
and set
\begin{equation}\label{falpha}
F(\a)\df (\a,f_\a).
\end{equation}
It is easy to see that $F$ has indeed become a faithful functor $\Fed\lra\DD$.

Next we prove that $F$ is full.
For $(A,C),(A\ap,C\ap)\in|\Fed|$, with $D\df F(A,C)$ and $D\ap\df F(A\ap,C\ap)$,
we consider $(\a,f)\in \DD(D,D\ap)$ and must show
\begin{equation}\label{ffalf}
f=f_\a.
\end{equation}
But since $f\circ\pi\ap=\pi\circ T(\a)=f_\a\circ\pi\ap$ and $\pi\ap$ is a surjection, (\ref{ffalf}) follows trivially.

Finally we confirm that $F$ is essentially surjective on objects.
Given $(A,\pi)\in|\DD|$ we know that $\pi:T(A)\lra X$ lies in $\PP$ and is, in particular, a quotient map. For the natural quotient map $\pi\ap:T(A)\lra T(A)/R_\pi$ we define $h:T(A)/R_\pi\lra X$ by $h(\pi\inv(x))\df x$, for all $x\in X$. Then $\pi=h\circ \pi\ap$ and, as it is well known, $h$ is a homeomorphism. Clearly, $\pi\ap$ is an irreducible mapping. Thus $(A,\pi\ap)\in|\DD|_{\rm nqm}$. With
 Corollary \ref{cor1f}, we obtain that there exists $(A,C)\in|\Fed|$ such that $F(A,C)=(A,\pi\ap)$. Obviously, $(A,\pi\ap)$ and $(A,\pi)$ are $\DD$-isomorphic.

In summary, the functor $F$ is an equivalence of categories.
\end{proof}

\begin{cor}\label{coriso}
The full subcategory $\DD_{\rm nqm}$ of the category $\DD$ with $|\DD_{\rm nmq}|\df |\DD|_{\rm nqm}$ is isomorphic to the category $\Fed$.
\end{cor}

\begin{proof}
The proof of Theorem \ref{thnewfed} shows that the functor $F:\Fed\lra\DD$ defined there actually takes values in $\DD_{\rm nmq}$.
Its restriction $F_1\df F\upharpoonright\Fed:\Fed\lra\DD_{\rm nmq}$ has an inverse, given
(in the notation of Lemma \ref{lem1}) by the assignments
$$(A,\pi)\mapsto (A,C_{(A,\pi)}) \mbox{ and }(\p,f)\mapsto \p.$$
Using Corollary \ref{cor1f}, one easily sees that these assignments are inverse to those of $F_1$, thus making $F_1$ bijective.
\end{proof}

\begin{nist}\label{n2n}
\rm

If $\a:A\lra B$ is a suprema-preserving Boolean homomorphism, then
$\a\in\Fed((A,\rho_s^A),(B,\rho_s^B))$ because in this case
condition (F)
is automatically fulfilled. Hence, we may define the functor $$J\ap:\AA\lra\Fed$$
by $J\ap(A)=(A,\rho_s^A)$ for every $A\in|\AA|$, and by
$J\ap(\a)=\a$ for every $\a\in\AA(A,B)$, which embeds $\AA$ fully into $\Fed$ and, obviously, satisfies  $J=F\circ J\ap$ (with $F$ as in the proof of Theorem \ref{thnewfed}).
Its image $\AA'$ is then a full subcategory of $\Fed$ isomorphic to $\AA$.
Setting $\Psi\df \hat{T}\circ F$ 
we obtain,
 using the definition of the contravariant functor $\hat{T}$ given in Proposition \ref{mere duality},
 that $$\Psi:\Fed\lra\CC$$ is a dual equivalence,
 with $\Psi(A,C)=T(A)/R_{(A,C)}$ for every $(A,C)\in|\Fed|$, $\Psi(\p)=f_\p$ for every $\p\in \Mor(\Fed)$ (see (\ref{falpha}) for the notation $f_\p$)
  and  $I\circ T=\Psi\circ J\ap.$
\end{nist}

\begin{pro}\label{prff}
For all $(A,C)\in|\Fed|$, the spaces  $\Phi^a(A,C)$ and $\Psi(A,C)$ are homeomorphic.
\end{pro}

\begin{proof}
For $(A,C)\in|\Fed|$ we
set
$h([u])\df \bigcup\{v\sbe A\st v\in T(A), v\in[u]\}$, for all $u\in T(A)$.  Then  $$h([u])= \bigcup\{v\sbe A\st v\in T(A), vR_{(A,C)}u\}=\{a\in A\st aCb \;{\rm for\; all}\; b\in u\}.$$ Indeed, the middle set is obviously contained in the set on the right, and the reversed inclusion follows from Theorem \ref{conclustth}. Using again Theorem \ref{conclustth}, we obtain that $h([u])$ is a cluster in $(A,C)$. Hence, the function
 $$h:\Psi(A,C)\lra \Phi^a(A,C)$$
 is well-defined, and Theorem \ref{conclustth} shows that $h$ is surjective. Also, in the notation of Corollary \ref{uniqult}, we have $h([u])=\s_u$, for all $u\in T(A)$.
 For showing the injectivity of $h$, we let $u,v\in T(A)$ with $[u]\neq[v]$, but suppose that $h([u])=h([v])$, so that $u\cup v\sbe h([u])$. This means $uR_{(A,C)}v$ and, thus, $[u]=[v]$, a contradiction. So, $h$ is a bijection.
 Setting $X\df \Psi(A,C)$, we let $\pi:T(A)\lra X$ be the natural quotient mapping. Then Lemma \ref{lem2} shows that $\pi$ is a closed irreducible mapping. Hence, by Theorem \ref{thalpon}, the map $\p_\pi:\CO(T(A))\lra \RC(X),\   H\mapsto \pi(H),$ is a Boolean isomorphism. Therefore, $\RC(X)=\{\pi(s_A(a))\st a\in A\}$. For every $a\in A$, we have that $\pi(s_A(a))=\{[u]\st a\in u\}$. Thus
$h(\pi(s_A(a)))=\{h([u])\st a\in u\}=\{h([u])\st a\in h([u])\}=\ups_{(A,C)}(a)$. Hence, $h$ is a homeomorphism.
\end{proof}

Note that Proposition \ref{prff} clarifies the definition of the contravariant functor $\Phi^a$ on objects. Next we compare the definitions of the contravariant functors $\Phi^a$ and $\Psi$ on the morphisms of the category $\Fed$.
\begin{nist}\label{fedmorphpsi}
\rm
For $\p\in\Fed((A,C),(A\ap,C\ap))$, we set $\pi\df\pi_{(A,C)},\;\pi\ap\df\pi_{(A\ap,C\ap)}$ (in the notation of Lemma \ref{lem2}) and obtain $\pi:T(A)\lra\Psi(A,C)$ and $\pi\ap:T(A\ap)\lra\Psi(A\ap,C\ap)$.
We then have $\Psi(\p)=f_\p=\pi\circ T(\p)\circ (\pi\ap)\inv$. With the homeomorphisms $h:\Psi(A,C)\lra\Phi^a(A,C)$ and $h\ap:\Psi(A\ap,C\ap)\lra\Phi^a(A\ap,C\ap)$ of Proposition \ref{prff} we can confirm the following assertion:

\medskip

\noindent{\bf Proposition.}\ \ $\Phi^a(\p)\circ h\ap=h\circ \Psi(\p)$.

\begin{proof}
Indeed, for all $u\ap\in T(A\ap)$,  one has
\medskip

\begin{tabular}{ll}
$(h\circ \Psi(\p))([u\ap])$&$=(h\circ f_\p)([u\ap])=(h\circ \pi\circ T(\p)\circ (\pi\ap)\inv)([u\ap])$\\
&$=h(\pi(T(\p)(u\ap)))=h([\p\inv(u\ap)])=\s_{\p\inv(u\ap)}.$\\
 \end{tabular}

\medskip
\noindent Also, $(\Phi^a(\p)\circ h\ap)([u\ap])=\Phi^a(\p)(\s_{u\ap})=\s_{\p\inv(u\ap)}$, so that $\Phi^a(\p)\circ h\ap=h\circ f_\p$ follows.
\end{proof}
\end{nist}


\begin{nist}\label{dprimdeq}
\rm
Being an equivalence of categories, the functor $F:\Fed\lra \DD$ of Theorem \ref{thnewfed} has an adjoint, both of whose composites with $F$ are naturally isomorphic to the corresponding identity functors. Since such adjoint functor is determined uniquely by $F$ only up to natural isomorphism, there is value in exhibiting a specific adjoint $F\ap:\DD\lra\Fed$ to $F$. We use the following notation: for a continuous surjection $f:X\lra Y$, we denote by $R_f$ the equivalence relation on $X$ determined by the fibres of $f$, by $q_f:X\lra X/R_f$ the natural quotient mapping, and by $h_f:X/R_f\lra Y$ the map with $h_f\circ q_f=f$. 
  We let $F_2:\DD_{\rm nqm}\hookrightarrow\DD$ be the inclusion functor, and the functor $F_2\ap:\DD\lra\DD_{\rm nqm}$ is defined by $F_2\ap(A,p)\df (A,q_p)$ for every $(A,p)\in|\DD|$, and $F_2\ap(\p,f)\df (\p,(h_p)\inv\circ f\circ h_{p\ap})$ for every $(\p,f)\in\DD((A,p),(A\ap,p\ap))$. 

\bigskip

\noindent{\bf Proposition.}\  $F_2\ap\circ F_2={\sf Id}_{\DD_{\rm nqm}}$ {\em and} $F_2\circ F_2\ap\simeq {\sf Id}_{\DD}$.

\begin{proof}
It is easy to see that $F_2\ap$ is a well-defined functor and that $F_2\ap\circ F_2={\sf Id}_{\DD_{\rm nqm}}$. Further, for every $(A,p)\in|\DD|$, we have $F_2(F_2\ap(A,p))=(A,q_p)$, and for every $(\p,f)\in\DD((A,p),(A\ap,p\ap))$, $F_2(F_2\ap(\p,f))=(\p,(h_p)\inv\circ f\circ h_{p\ap})$. Then, obviously, $(1_A,h_p):(A,p)\lra F_2(F_2\ap((A,p)))$ is a $\DD$-isomorphism and $(1_{A\ap},h_{p\ap})\circ(\p,f)=F_2(F_2\ap(\p,f))\circ(1_A,h_p)$. Therefore, $F_2\circ F_2\ap\simeq {\sf Id}_{\DD}$.
\end{proof}

In the notation of Corollary \ref{coriso}, one has $F=F_2\circ F_1$. Then, with $F\ap\df F_1\ap\circ F_2\ap$,  we obtain
$$F\ap:\DD\lra\Fed,\ \ F\circ F\ap\simeq {\sf Id}_{\DD}, \mbox{ and }\ F\ap\circ F\simeq {\sf Id}_{\Fed}.$$
One also easily confirms:

\medskip

\noindent{\bf Fact.}\ $F\ap\circ J=J\ap$.

\medskip

By Proposition \ref{n1n}, $\BB$ is a coreflective subcategory of $\CC$. Hence, we can use both, Proposition \ref{mere duality} and Theorem \ref{what about S}. Thus, with the adjoint $\tilde{S}$ of $\tilde{T}$ defined there, putting

$$\Psi\ap\df F\ap\circ \ti{S}:\;\HC_{{\rm qop}}\lra\Fed,$$
we immediately obtain (using the above Fact, Theorem \ref{what about S}, and the definition of $\Psi$ of \ref{n2n})
$$J\ap\circ S=\Psi\ap\circ I,\ \ \Psi\ap\circ\Psi\simeq {\sf Id}_{\Fed}\ \ \mbox{ and }\ \ \Psi\circ\Psi\ap\simeq {\sf Id}_{\CC}.$$
\end{nist}

We can now analyze the connection between the contravariant functors $\Psi\ap$ and $\Phi^t$. Recall that,
for a compact Hausdorff space $X$, $\pi_X:EX\lra X$ denotes the projective cover of $X$ (see \ref{niab}(b)); also,  $\p_{\pi_X}:\CO(EX)\lra\RC(X),\ P\mapsto\pi_X(P),$ is the Boolean isomorphism of Theorem \ref{thalpon}. Finally, if $f\in \HC_{{\rm qop}}(X,X\ap)$, then $Ef$ denotes the unique continuous mapping $\hat{f}:EX\lra EY$ such that $f\circ \pi_X=\pi_Y\circ\hat{f}$ (see Propositions \ref{HJTHM} and \ref{begen}).

\begin{pro}\label{psiapfedtob}
{\em (a)} For every compact Hausdorff space $X$, the CNCAs $\Psi\ap(X)$ and $\Phi^t(X)$ are CA-isomorphic (and, thus, $\Fed$-isomorphic).

\medskip

\noindent{\em (b)} For every $f\in\HC_{{\rm qop}}(X,X\ap)$, $\Phi^t(f)=\p_{\pi_X}\circ\Psi\ap(f)\circ(\p_{\pi_{X\ap}})\inv$.
\end{pro}

\begin{proof}
(a) Let $X\in|\CC|$. Then, by Theorem \ref{what about S}, $\ti{S}(X)=(\CO(EX), p)$, where $p=\pi_X\circ (t_{EX})\inv$, with $t_{EX}:EX\lra T(\CO(EX))$ as defined in \ref{Std}. For $A\df \CO(EX)$ and $\pi\df\pi_X$ we have $\Psi\ap(X)=F\ap(\ti{S}(X))=F\ap(A,p)=(A,C_{(A,p)})$. Also,  $\Phi^t(X)=(\RC(X),\rho_X)$  (see Example \ref{rct} for $\rho_X$).  We now show that

$$\p_\pi:\Psi\ap(X)\lra\Phi^t(X)$$ is a CA-isomorphism. Indeed, for $P,Q\in A$, we have, by definition of the relation $C_{(A,p)}$ (see Lemma \ref{lem1}),

\begin{tabular}{ll}
$PC_{(A,p)}Q$&$\Leftrightarrow p(s_A(P))\cap p(s_A(Q))\nes$\\
&$\Leftrightarrow\pi((t_{EX})\inv(s_A(P)))\cap \pi((t_{EX})\inv(s_A(Q)))\nes$\\
&$\Leftrightarrow \pi(\{\bigcap u\!\st\!u\!\in\! \Ult(A),P\in u\})\cap \pi(\{\bigcap v\!\st\! v\!\in \!\Ult(A),Q\in v\})\nes$\\
&$\Leftrightarrow (\ex x\in P)(\ex y\in Q)(\pi(x)=\pi(y))$\\
&$\Leftrightarrow \pi(P)\cap\pi(Q)\nes\Leftrightarrow \p_\pi(P)\rho_X\p_\pi(Q).$\\
\end{tabular}

\noindent Therefore, $\Psi\ap(X)$ and $\Phi^t(X)$ are CA-isomorphic.
\medskip

\noindent(b) For $f\in\HC_{{\rm qop}}(X,X\ap)$, we have $\ti{S}(f)=(S^t(Ef),f)$ (Theorem \ref{what about S}). Hence, $\Psi\ap(f)=F\ap(\ti{S}(f))=S^t(Ef)$; also, $\Phi^t(f)(H\ap)=\cl_X(f\inv(\int_{X\ap}(H\ap)))$, for every $H\ap\in\RC(X\ap)$ (see Theorem \ref{dvth} and its proof, and the definition of $\Phi^t$ in \ref{dcompn}). Now, for $H\ap\in\RC(X\ap)$, since $Ef$ is an open mapping (by Proposition \ref{begen}), and since $f\inv=\pi_X\circ(Ef)\inv\circ(\pi_{X\ap})\inv$, with Theorem \ref{thalpon} and \cite[Ex. 1.4.C]{E} we obtain

\begin{tabular}{ll}
$(\p_{\pi_X}\circ\Psi\ap(f)\circ(\p_{\pi_{X\ap}})\inv)(H\ap)$&$=\p_{\pi_X}(\Psi\ap(f)((\p_{\pi_{X\ap}})\inv(H\ap)))$\\
&$=\p_{\pi_X}((Ef)\inv(\cl_{EX\ap}(\pi_{X\ap}\inv(\int_{X\ap}(H\ap)))))$\\
&$=\pi_X(\cl_{EX}((Ef)\inv(\pi_{X\ap}\inv(\int_{X\ap}(H\ap)))))$\\
&$=\cl_X(\pi_X((Ef)\inv(\pi_{X\ap}\inv(\int_{X\ap}(H\ap)))))$\\
&$=\cl_X(f\inv(\int_{X\ap}(H\ap)))=\Phi^t(f)(H\ap)$.\\
\end{tabular}

Therefore, $\Phi^t(f)=\p_{\pi_X}\circ\Psi\ap(f)\circ(\p_{\pi_{X\ap}})\inv$.
\end{proof}

\bigskip

\noindent{\bf Acknowledgements.} The authors would like to thank Prof. V. Valov for the informative discussions on
the subject of this paper.

\end{document}